\documentclass[10pt, english]{smfart}
\usepackage{amssymb}
\usepackage{amsmath}
\usepackage{amsthm}
\usepackage{mathrsfs}
\usepackage{frcursive}
\usepackage{graphicx}
\usepackage{color}
\usepackage{manfnt}
\usepackage[ansinew]{inputenc}
\usepackage[all]{xy}
\usepackage{hyperref}
\usepackage{pdfsync}
\newtheorem{axiom}{Axiom}
\newtheorem{thm}{Theorem}[section]
\newtheorem{cor}[thm]{Corollary}

\newtheorem{lem}[thm]{Lemma}
\newtheorem{prop}[thm]{Proposition}

\theoremstyle{definition}
\newtheorem{rem}[thm]{Remark}

\newtheorem*{thm*}{Theorem}

\newtheorem{example}[thm]{Example}

\newtheorem{notations}[thm]{Notations}

\theoremstyle{remark}
\numberwithin{equation}{section}

\newcommand{\abs}[1]{\left\vert#1\right\vert}

\newcommand{\Qp}{\mathbb Q_{p}}
\newcommand{\ord}{\mathrm{ord}\:}
\newcommand{\ac}{\overline{\mathrm{ac}}\:}

\newcommand{\Def}{\mathrm{Def}}

\newcommand{\A}{\mathbb A}
\newcommand{\Kdim}{\mathrm{Kdim}\:}

\newcommand{\Z}{\mathbb Z}

\renewcommand{\L}{\mathbb L}

\newcommand{\ra}{\rightarrow}

\newcommand{\RDef}{\mathrm{RDef}}

\newcommand{\C}{\mathcal C}

\newcommand{\Jac}{\text{Jac}\:}

\def\llp{\mathopen{(\!(}}

\def\rrp{\mathopen{)\!)}}

\def\Id{\mathrm{Id}}

\def\cartesien{%
    \ar@{-}[]+R+<6pt,-1pt>;[]+RD+<6pt,-6pt>%
    \ar@{-}[]+D+<1pt,-6pt>;[]+RD+<6pt,-6pt>%
  }

\definecolor{vert}{rgb}{0,0.6,0.2}

\title[Commutativity of pull-back and push-forward functors]
{On the commutativity of pull-back and push-forward functors on motivic constructible functions}
\author{Jorge Cely}
\address{Bogot\'a, Colombia}
\email{celyje@gmail.com}
\author{Michel Raibaut}
\address{Univ. Grenoble Alpes, Univ. Savoie Mont Blanc, CNRS, LAMA}
\address{Le Bourget-du-lac, 73376, France.}
\email{Michel.Raibaut@univ-smb.fr}
\urladdr{raibautm.perso.math.cnrs.fr/site/MichelRaibaut.html}

\begin{document}
\maketitle
\begin{abstract}
	In this article, we study the commutativity between the pull-back and the push-forward functors on constructible functions in Cluckers--Loeser motivic integration.
\end{abstract}
\date{\today}

\section*{Introduction}

Let $k$ be a characteristic zero field. By analogy with integration over local fields, Kontsevich introduced in \cite{Kon95a} an integration theory on finite dimensional vector spaces over $k((t))$, called \emph{motivic integration}, with values in a Grothendieck ring of varieties $K_0(Var_k)$. Later, Cluckers -- Loeser in \cite{CluLoe08a} (announced in \cite{CluLoe04a}, \cite{CluLoe04b}) generalized that construction allowing in particular integrals with parameters in the context of henselian valued fields of equal characteristic zero, and in \cite{HruKaz06}, Hrushovski -- Kazhdan treated the case of algebraically closed valued fields of equal characteristic zero. 

For any \emph{definable sets} $X$, given by a first order formula in the Denef-Pas language, Cluckers -- Loeser construct in \cite{CluLoe08a} an algebra $C(X)$ of \emph{constructible motivic functions} defined on $X$, and in \cite{CluLoe05b} and \cite{CluLoe10a} they enlarge it in an algebra $C(X)^{exp}$ of \emph{exponential constructible functions}.  
Moreover, for any definable function $f:X\to Y$, they define a pull-back functor $f^* : C(Y)^{exp} \to C(X)^{exp}$, an abelian subgroup $I_YC(f)^{exp}$ of $C(X)^{exp}$ of $f$-integrable constructible functions and a push-forward functor $f_! : I_YC(f)^{exp} \to C(Y)^{exp}$ which corresponds to an integration along fibers of $f$. 
Roughly speaking, for any definable set $X$, there exists integers $m$, $n$ and $r$ such that for any $k$-field extension $K$, the set of $K$-rational points $X(K)$ is contained in $K((t))^m\times K^n \times \mathbb Z^r$. Such definable set admits a cell decomposition and similarly to the construction of the integration against Euler characteristic in the real semi-algebraic setting, the construction of the functor $f_!$ is given by an induction process on the valued field dimension.
We recall in the first section of this article main ideas, definitions and results of these constructions. 

In \cite{WF} the second author introduced a notion of \emph{definable distributions} in Cluckers--Loeser motivic setting. He introduced also a notion of \emph{motivic wave front set}, which allows him, as in the real setting \cite{Hormander83} or in the p-adic setting \cite{Hei85a} and \cite{UD}, to study the pull-back of a distribution by a function
which requires the natural following commutativity relation between pull-back and push-forward functors

\begin{thm*} Let $X$, $W$ and $W'$ be definable sets over $k$, let $\gamma$ be a definable morphism from $W$ to $W'$. 
	We denote by $\pi_W$ the projection from $W\times X$ to $W$ and by $\pi_{W'}$ the projection from $W' \times X$ to $W'$. Let $[\varphi]$ be a constructible exponential function in $C(W'\times X)^{exp}$.
	\begin{enumerate}
		\item   If $[\varphi]$ is $\pi_{W'}$-integrable then 
			$[(\gamma \times Id_X)^*\varphi]$ is $\pi_W$-integrable.
			Furthermore, if $\gamma$ is onto then this implication is an equivalence.
		\item  If $[\varphi]$ satisfies the condition (\ref{thm:integrability-condition}) then  
			$$ \pi_{W!}\left[(\gamma\times Id_X)^*\varphi \right] = \gamma^*(\pi_{W'!} [\varphi]).$$
	\end{enumerate}

\end{thm*}

In this article, we start in section \ref{survey} by recalling main definitions and ideas of Cluckers--Loeser motivic integration. Then, in section \ref{proof}, we prove above theorem in a slightly more general context (lemma \ref{projlem} and theorem \ref{mainthm}) following all different steps of the construction of the theory as the induction process on the valued field dimension using cell decompositions and the computation at the residue and value group levels.

\section{Motivic integration and constructible motivic functions} \label{survey}
For the reader's convenience
we shall start by recalling briefly some definitions, notations and constructions from \cite{CluLoe08a} and
\cite{CluLoe10a} that will be used in this article.  For an introduction to this circle of ideas  we refer to the surveys \cite{CluLoe05a}, \cite{CluHalLoe11} and \cite{GorYaf09} and the notes \cite{CluLoe04a}, \cite{CluLoe04b} and \cite{CluLoe05b}. 
\subsection{Denef-Pas, Presburger language}
We fix a field $k$ of characteristic zero and we denote by $\mathrm{Field}_{k}$ the category of fields containing $k$.
For any field $K$ in this category we consider the field of Laurent series
$K(\!(t)\!)$ endowed with its natural \emph{valuation}
$$\ord : K(\!(t)\!)\setminus\{0\} \longrightarrow \mathbb Z$$
extended by $\ord 0=+\infty$,
and with the \emph{angular component} mapping
$$\ac:K(\!(t)\!)\ra K$$
defined by $\ac(x)=xt^{-\ord x}\mod t$ if $x\neq 0$ and $\ac (0) =0$.\\

We shall use the three sorted language introduced by Denef and Pas in \cite{Pas89}
$$\mathscr L_{DP,P} = (\mathbf{L}_{\mathrm{Val}},\mathbf{L}_{\mathrm{Res}},\mathbf{L}_{\mathrm{Ord}},\ord, \ac)$$
with sorts corresponding respectively to \emph{valued field}, \emph{residue field} and
\emph{value group} variables.
The languages $\mathbf{L}_{\mathrm{Val}}$ and $\mathbf{L}_{\mathrm{Res}}$ are the ring language
$\mathbf{L}_{\mathrm{Rings}}=(+,-, \cdot ,0,1)$ and the language $\mathbf{L}_{\mathrm{Ord}}$ is the Presburger language
$$\mathbf{L}_{\mathrm{PR}}=\{+,-,0,1,\leq\}\cup\{\equiv_{n}\mid n\in \mathbb N, n>1\},$$
with $\equiv_{n}$  symbols interpreted as  equivalence relation modulo $n$. Symbols $\ord$and $\ac$will be interpreted respectively as valuation and angular component, so that for any $K$ in $\mathrm{Field}_{k}$ the triple $(K(\!(t)\!),K,\mathbb Z)$ is a structure for $\mathscr L_{DP,P}$. We shall also add
constant symbols in the $\mathrm{Val}$-sort and in the $\mathrm{Res}$-sort for elements of $k(\!(t)\!)$, resp.~of $k$.\\

We will work with the $\mathcal{L}_{DP,P}$-theory $H_{\ac,0}$ of  structures whose valued field is Henselian, with characteristic zero residue field , and with value group $\Z$.
Denef and Pas proved in \cite{Pas89} the following theorem on elimination of valued field quantifiers.

\begin{thm}[Denef-Pas \cite{Pas89}, Presburger \cite{Presb29}]  \label{QE}
Every formula $\phi(x,\xi,\alpha)$ without parameters in the $\mathscr L_{DP,P}$-language, with $x$ variables in the $\mathrm{Val}$-sort, $\xi$ variables in the $\mathrm{Res}$-sort and $\alpha$ variables in the
$\mathrm{Ord}$-sort is $H_{\ac,0}$-equivalent to a finite disjunction of formulas of the form
$$\psi(\ac f_{1}(x),...,\ac f_{k}(x),\xi)\land \eta(\ord f_1(x),...,\ord f_k(x),\alpha),$$
with $\psi$ a $\mathbf{L}_{\mathrm{Res}}$-formula, $\eta$ a $\mathbf{L}_{\mathrm{Ord}}$-formula without quantifiers and $f_1,...,f_k$ polynomials in $\Z[x]$. The theory $H_{\ac,0}$ admits elimination of quantifiers in the valued field sort.
\end{thm}

\subsection{Definable subassignments}

From now on we will work with the Denef-Pas language enriched with constant symbols in the $\mathrm{Val}$-sort and in the $\mathrm{Res}$-sort for elements of $k(\!(t)\!)$, resp.~of $k$, and we will denote this language also by $\mathcal{L}_{DP,P}$.

\subsubsection{Definable subassignments and definable morphisms}
Let $\varphi$ be a formula respectively in $m$, $n$ and $r$ free variables in the various sorts. For every field $K$ in
$\mathrm{Field}_{k}$, we denote by $h_{\varphi}(K)$ the subset of
$$h[m,n,r](K):=K(\!(t)\!)^{m}\times K^{n} \times \mathbb Z^{r}$$
consisting of points satisfying $\varphi$. The assignment $K\mapsto h_{\varphi}(K)$ is called a
\emph{definable subassignment} or \emph{definable set}. For instance we will denote by $\{*\}$ the definable subassignment $h[0,0,0]$ defined by $K \mapsto \mathrm{Spec}\: K$.
A \emph{definable morphism} $F$ between two definable subassignments $h_{\varphi}$
and $h_{\psi}$ is a collection of applications parametrized by $K$ in $\mathrm{Field}_{k}$ $$F(K):h_{\varphi}(K)\ra h_{\psi}(K)$$ such that the graph map
$K\mapsto \mathrm{Graph}F(K)$ is a definable subassignment. Definable subassignments and definable morphisms are precisely objects
and morphisms of the category of definable subassignments over $k$ denoted by $\mathrm{Def}_{k}$.
More generally, for any definable subassignment $S$ in $\Def_{k}$, we will consider the category $\Def_S$ of definable subassignments over $S$ whose objects are definable morphisms
$\theta_{Z}$ in $\Def_{k}$ from a definable $Z$ to $S$ and morphisms are definable maps $g:Y\ra Z$ such that
$\theta_Y = \theta_Z \circ g$.
Sometimes, instead of using $\theta_Z$, we will simply say that $Z$ is a definable set in $\Def_S$.

\subsubsection{Finiteness of some definable functions}
We deduce from theorem \ref{QE} on quantifier elimination the following corollary 
\begin{cor} \label{fini}
For non negative integers $m$,$n$ and $r$, every definable map from $h[0,n,0]$ to $h[m,0,r]$ or from $h[0,0,r]$ to $h[m,n,0]$ or from $h[0,n,r]$ to $h[m,0,0]$ takes finitely many values.
\end{cor}

\subsubsection{Points and fibers}
A \emph{point} $x$ of a definable set $X$ is by definition a couple $x=(x_{0},K)$ where $K$ is an extension of $k$ and $x_{0}$ is a point of $X(K)$. The field $K$ will be denoted by $k(x)$ and called
\emph{residue field} of $x$.

Let $f$ be a definable morphism from a subassignment of $h[m,n,r]$ denoted by  $X$ to a subassignment of $h[m',n',r']$ denoted by $Y$. Let $\varphi(x,y)$ be the formula which describes the graph of $f$, where $x$ runs over $h[m,n,r]$ and $y$ runs over $h[m',n',r']$.
For every point $y=(y_{0},k(y))$ of $Y$, the \emph{fiber} $X_y$ is the object of $\Def_{k(y)}$ defined by the formula $\varphi(x,y_0)$ which has coefficients in $k(y)$ and $k(y)(\!(t)\!)$. Taking fibers at $y$ gives rise to a functor $i_{y}^{*}:\Def_{Y}\ra \Def_{k(y)}$.

\subsubsection{Dimension}
For any positive integer $m$, an algebraic subvariety $Z$ of $\A_{k(\!(t)\!)}^{m}$ induces a definable subassignment $h_{Z}$ of $h[m,0,0]$ with $h_{Z}(K)$ equal to $Z(K(\!(t)\!))$ for any extension $K$ of $k$.
The \emph{Zariski closure} of a subassignment $S$ of $h[m,0,0]$ is by definition the subassignment of the intersection $W$ of all algebraic subvarieties $Z$ of $\A_{k(\!(t)\!)}^{m}$ such that $h_{Z}$ contained $S$. The \emph{dimension} $\Kdim S $ of $S$ is naturally defined as $\dim W$.
More generally, the dimension of a subassignment $S$ of $h[m,n,r]$ is defined as the dimension $\Kdim p(S)$ where $p$ is
the projection from $h[m,n,r]$ to $h[m,0,0]$.\\  It is  proved in \cite{CluLoe08a},
using results of Pas \cite{Pas89} and van den Dries \cite{VDDries89}, that
 isomorphic definable subassignments in $\Def_{k}$ have the same dimension.

\subsection{Grothendieck rings and exponentials}

\subsubsection{The category $\RDef_{k}^{\:\exp}$}
For any definable subassignment $Z$ in $\Def_{k}$,
the  subcategory $\RDef_Z$ of $\Def_{Z}$ whose
objects are definable morphisms $\pi_Y$, with $Y$ a subassignment of a product
$Z\times h[0,n,0]$, $n$ a non negative integer and $\pi_Y$ the canonical projection on $Z$,
has been introduced in \cite{CluLoe08a}.

\begin{example}
If $Z$ is the point $h[0,0,0]$, then the subcategory $\RDef_Z$ is the category of definable sets in the ring language with coefficients from $k$.
\end{example}

More generally, in \cite{CluLoe10a} motivic additive characters were considered in this context through the category $\RDef_{Z}^{\:\exp}$ whose objects are triples $(\pi_Y,\xi, g)$ with $\pi_Y$ a definable set in $\RDef_{Z}$, $\xi$ a definable morphism from $Y$ to $h[0,1,0]$ and $g$ a definable morphism from $Y$ to $h[1,0,0]$. A morphism from $(\pi_{Y'},\xi', g')$ to $(\pi_Y,\xi, g)$ in
$\RDef_{Z}^{\:\exp}$ is a morphism $h$ from $Y'$ to $Y$ satisfying the equalities
$$\pi_{Y'}=\pi_{Y}\circ h,\:\: \xi'=\xi \circ h,\:\:g'= g \circ h.$$

\begin{rem}\label{rem:1.4}The functor $\pi_Y \mapsto (\pi_Y,0,0)$
allows to identify
 $\RDef_{Z}$ as  a full subcategory of $\RDef_{Z}^{\:\exp}$. \end{rem}

\subsubsection{The Grothendieck ring $K_{0}(\RDef_{Z}^{\:\exp})$}  \label{exponential}
As an abelian group it is the free abelian group over symbols $[\pi_Y,\xi,g]$
modulo the following relations :

 \emph{Isomorphism}.
 For any isomorphic $(\pi_Y,\xi, g)$ and $(\pi_{Y'},\xi', g')$,
 \begin{equation} \tag{R1}
  [\pi_Y,\xi,g]=[\pi_{Y'},\xi',g']
 \end{equation}

\emph{Additivity}.
 For $\pi_Y$ and $\pi_{Y'}$ definable subassignments of some $\pi_X$ in $\RDef_{Z}$ and for $\xi$ and $g$ defined on $Y\cup Y'$
 \begin{equation} \tag{R2}
 [\pi_{Y\cup Y'},\xi,g]+[\pi_{Y\cap Y'},\xi_{\mid Y \cap Y'},g_{\mid Y \cap Y'}]
 =[\pi_{Y},\xi_{\mid Y},g_{\mid Y}]+[\pi_{Y'},\xi_{\mid Y'},g_{\mid Y'}].\end{equation}

 \emph{Compatibility with reduction}.
 For any $\pi_Y$ in $\Def_Z$, for any definable morphism $f$ from $Y$ to $h[1,0,0]$ with $\ord f(y) \geq 0$ for any $y\in Y$  we have
 \begin{equation} \label{R3} \tag{R3}
  [\pi_Y,\xi,g+f]=[\pi_Y,\xi+\overline{f},g]
 \end{equation}
 with $\overline{f}$ the reduction
 of $f$ modulo $(t)$.

 \emph{Sum over the line}.
 Let $p$ be the canonical projection from $Y[0,1,0]$ to $h[0,1,0]$. If the morphisms $\pi_{Y[0,1,0]}$, $g$ and $\xi$ all factorize through the canonical projection from $Y[0,1,0]$ to $Y$ then
 \begin{equation} \label{R4} \tag{R4}
 [Y[0,1,0]\ra Z,\xi+p,g]=0.
 \end{equation}

 This Grothendieck group is endowed with a ring structure by setting
\begin{equation} \label{R5} \tag{R5}
 [\pi_Y,\xi,g] \cdot [\pi_{Y'},\xi',g']=[\pi_{Y\otimes_Z Y'},\xi\circ p_Y+\xi'\circ p_Y',g\circ p_Y+g'\circ p_Y']
\end{equation}
where $Y\otimes_Z Y'$ is the fiber product of $Y$ and $Y'$ above $Z$, $p_Y$ is the projection to $Y$ and $p_{Y'}$ is the projection to $Y'$. The element $[\Id_Z,0,0]$ is the multiplicative unit of $K_{0}(\RDef_Z^{\:\exp})$.
The Grothendieck ring $K_{0}(\RDef_Z)$ is defined as above and the functor defined in Remark \ref{rem:1.4} induces an injection $K_{0}(\RDef_Z) \to K_{0}(\RDef_Z^{\:\exp})$.

\begin{rem}
The element $[\pi_Y,\xi, g]$ of the Grothendieck ring $K_{0}(\RDef_{Z}^{\:\exp})$ will be denoted by $e^{\xi}E(g)[\pi_Y]$.
We will abbreviate $e^{0}E(g)[\pi_Y]$ by $E(g)[\pi_Y]$, $e^{0}E(0)[\pi_Y]$ by $[\pi_Y]$ and $e^{0}E(g)[\Id_Z]$ by $E(g)$.
\end{rem}

\begin{rem}[Interpretation of $E$] The element $E(g)$ in $K_{0}(\RDef_{Z}^{\:\exp})$ can be viewed as the exponential  (at the valued field level of the definable morphism $g$ from $Z$ to $h[1,0,0]$, said otherwise, it is a motivic additive character on the valued field evaluated in $g$. More precisely, by relations (\ref{R3}) and (\ref{R5}), $E$ can be interpreted as a universal additive character  which is trivial on the maximal ideal of the valuation ring. This is compatible with specialization to $p$-adic fields as explained in  Section 9 of \cite{CluLoe10a}.
\end{rem}

\begin{rem}[Interpretation of $e$]
 The element $e(\xi)$ in $K_{0}(\RDef_{Z}^{\:\exp})$ can be considered as the exponential (at the residue field level) of the definable morphism $\xi$ from $Z$ to $h[0,1,0]$.  By relation (\ref{R4}), $e$ can  be interpreted as a universal additive character on the residue field. For instance in the case where $Z$ is the point, the relation $[h[0,1,0]\ra \{*\},p,0] = 0$ should be interpreted as an abstraction of the classical nullity of the sum of a non trivial character over elements of a finite field. Relation (\ref{R3}) expresses compatibility under reduction modulo the uniformizing parameter between the exponentials over the valued field and over the residue field.
\end{rem}

\subsubsection{Pull-back and push-forward} \label{Res-pull-back-push-forward}
For $f:Z\to Z'$ in $\Def_k$ there is a natural pull-back morphism $f^*:K_0(\RDef_{Z'}^{\:\exp}) \to K_0(\RDef_{Z}^{\:\exp})$, induced by the fiber product. Furthermore, if $f$ is a morphism in $\RDef_{Z'}$, then composition with $f$ induces a morphism 
$f_{!}:K_0(\RDef_{Z}^{\:\exp}) \to K_0(\RDef_{Z'}^{\:\exp})$.

\subsection{Constructible exponential functions}

\subsubsection{Constructible motivic functions}
In the $p$-adic context (see \cite{Denef84}, \cite{Igusabook}, \cite{CluLoe10a} and \cite{CluGorHal14a}), for instance over the field $\mathbb Q_{p}$ itself, one fixes an additive character $\Psi : K \ra \mathbb C^\times$ trivial on $p\mathbb Z_p$ and non trivial on the set $\ord x =0$ and one denotes by $\mathbb A_{p}$ the ring
$\mathbb Z[1/p, 1/(1-p^{-i})]$.
For any $X$ contained in some $\mathbb Q_p^m$ and definable for the Macintyre language, it is natural to define the $\mathbb A_{p}$-algebra of constructible functions on $X$ denoted by $\mathcal C(X)$ and generated by function of the form
 $\abs{f}\ord(h)$ where $f$ and $h$ are definable functions from $X$ to $\mathbb Q_p$ and $h$ does not vanish. In \cite{CluLoe10a}, also a variant with additive characters is introduced, called constructible exponential functions on $X$ and denoted by $\mathcal C(X)^{\exp}$. The algebra $\mathcal C(X)^{\exp}$ is generated by $\mathcal C(X)$  and functions of the form $\psi(g)$ with $g:X\to\Qp$ with $\psi$ a nontrivial additive character on $\Qp$.

Analogously, in \cite{CluLoe08a} Cluckers and Loeser consider the ring
$$\mathbb A=\mathbb Z\left[\L,\L^{-1},\left(\frac{1}{1-\L^{-i}}\right)_{i>0}\right],$$
where $\L$ is a symbol, and they define the ring $\mathcal C(Z)$ of
\emph{constructible motivic functions} on a definable set $Z$ by
$$\mathcal C(Z):=K_{0}(\RDef_{Z}) \otimes_{\mathcal{P}^{0}(Z)} \mathcal{P}(Z),$$
where $\mathcal P(Z)$, called ring of \emph{Presburger constructible functions}, is the subring of the ring of functions from the set of points of $Z$ to $\mathbb A$, generated
by constant functions, definable functions from $Z$ to $\mathbb Z$ and functions of the form $\L^{\beta}$ with
$\beta$ a definable function from $Z$ to $\mathbb Z$. Here, $\mathcal{P}^{0}(Z)$ is the subring of $\mathcal P(Z)$ generated by the constant function $\L$ and the characteristic functions $1_{Y}$ of definable subsets $Y$ of the base $Z$.
The tensor product is given by the morphism from $\mathcal{P}^{0}(Z)$ to $K_{0}(\RDef_{Z})$ sending $1_{Y}$ to the class $[Y \ra Z]$
of the canonical injection from $Y$ to $Z$ and sending $\L$ to the class $[Z[0,1,0]\ra Z]$ of the canonical projection to $Z$.

The following proposition \cite[Proposition 5.3.1]{CluLoe08a} allows to dissociate the $\text{Res}$-sort with the value group sort.
\begin{prop} \label{prop:dissociation}
	Let $S$ be a definable subassignment of $h[0,n,0]$.
	\begin{itemize}
		\item Let $W$ be a definable subassignment of $h[0,n,0]$. The canonical morphism
			$$\mathcal P(S) \otimes_{\mathcal P^{0}(S)} K_{0}(\RDef_{S\times W}) \ra \mathcal C(S \times W) $$
			is an isomorphism.
		\item Let $W$ be a definable subassignment of $h[0,0,r]$. The canonical morphism 
			$$K_{0}(\RDef_S)\otimes_{\mathcal P^{(0)}(S)} \mathcal P(S\times W) \ra \mathcal C(S \times W)$$
			is an isomorphism.
	\end{itemize}
\end{prop}

\subsubsection{Constructible exponential functions} \label{def:cef}
For any definable set $Z$ in $\Def_{k}$, the ring $\mathcal C(Z)^{\:\exp}$ of \emph{constructible exponential functions} is defined in
\cite{CluLoe10a} by
$$\mathcal C(Z)^{\:\exp}:=\mathcal C(Z)\otimes_{K_{0}(\RDef_{Z})}K_{0}(\RDef_{Z}^{\:\exp}),$$
where we use the morphism $a \mapsto a \otimes 1_Z$ from $K_{0}(\RDef_{Z})$ to $\mathcal C(Z)$.
For any integer $d$, we denote by $\mathcal C^{\leq d}(Z)^{\:\exp}$ the ideal of \emph{constructible functions of $K$-dimension $\leq d$}, namely the ideal generated by the characteristic functions $1_{Z'}$ of subassignments $Z'$ of $Z$ of dimension at most $d$. A constructible function has \emph{$K$-dimension $d$}, if it belongs to $\mathcal C^{\leq d}(Z)^{\:\exp} \setminus \mathcal C^{\leq d-1}(Z)^{\:\exp}$. 
This family of ideals is a filtration of the ring $\mathcal C(Z)^{\:\exp}$ and
the graded ring associated
$$C(Z)^{\:\exp} = \oplus_{d\in \mathbb N}\:\mathcal C^{\leq d}(Z)^{\:\exp}/\mathcal C^{\leq d-1}(Z)^{\:\exp}$$
is called ring of \emph{constructible exponential Functions}.

\begin{rem}
 Constructible Functions can be compared to the equivalence classes of Lebesgue measurable functions (equality up to a zero measure set). In this article we will just write function for Function; the difference still being visible in the notation $C(Z)^{\:\exp}$ versus $\mathcal C(Z)^{\:\exp}$.
\end{rem}

 \subsection{Cell decomposition}  \label{celldecomppart}

  Let $C$ be a definable subassignment in $\Def_{k}$. Let $\alpha : C \ra \Z$, $\xi : C\ra h[0,1,0]\setminus \{0\}$ and $c:C\ra h[1,0,0]$ be definable morphisms.\\

  $\bullet$ The cell $Z_{C,c,\alpha,\xi}$  with basis $C$, center $c$, order $\alpha$ and angular component $\xi$, is
  $$
  Z_{C,c,\alpha,\xi} =
  \left\{
  (y,z) \in C[1,0,0]\:
  \begin{array}{|l}
   \ord(z-c(y))=\alpha(y) \\
   \ac(z-c(y))=\xi(y)
  \end{array}
  \right\}
  $$
Note that this definable set is a family of balls $B(c(y)+\xi(y)t^{\alpha(y)},\alpha(y)+1)$ parametrized by the base $C$. The axiom (Axiom \ref{volume-boule}), below gives the push-forward morphism corresponding to the projection of this cell on its base $C$, that is, integration in the fibers of this projection map.\\

$\bullet$ The cell $Z_{C,c}$ with basis $C$ and center $c$ is
$$
  Z_{C,c} =
  \left\{
  (y,z) \in C[1,0,0]\mid z=c(y)
  \right\}.
  $$
  The change of variables formula (theorem \ref{cov}) gives in particular, the push-forward morphism corresponding to the projection of that cell on its base.\\

More generally, a definable subassignment $Z$ of $S[1,0,0]$ for some $S$ in $\Def_{k}$ is called a \emph{$1$-cell} or a \emph{$0$-cell} if there exists a definable isomorphism
$$\lambda : Z\ra Z_{C,c,\alpha,\xi}\subset S[1,s,r]\:\:\mathrm{or}\:\: \lambda : Z\ra Z_{C,c}\subset S[1,s,0],$$
called \emph{presentation} of the cell $Z$,
where the base $C$ is contained in $S[0,s,r]$ and such that the morphism $\pi \circ \lambda$  is the identity on $Z$ with $\pi$ the projection to $S[1,0,0]$.

Let us state  a variant of Denef-Pas Cell Decomposition theorem \cite{Pas89}, theorem 7.2.1 of \cite{CluLoe10a}, that will be used in the proof of the projection lemma \ref{projlem} in subsection \ref{cas-sans-exp}.

\begin{thm}[Cell decomposition]  \label{celldecompthm}
 Let $X$ be a definable subassignment of $S[1,0,0]$ with $S$ in $\Def_{k}$.
 \begin{enumerate}
  \item The definable set $X$ is a finite disjoint union of cells.
  \item For every $\varphi$ in $\mathcal C(X)$ there exists a finite partition of $X$ into cells $Z_{i}$ with presentation
  $(\lambda_{i},Z_{C_{i}})$ such that
  $\varphi_{\mid Z_{i}}=\lambda_{i}^{*}p_{i}^{*}(\psi_{i}),$ with
  $\psi_{i}$ in $\mathcal C(C_{i})$ and $p_{i}:Z_{C_{i}} \ra C_{i}$ the projection. Similar statements hold for $\varphi$
  in $\mathcal P(X)$ and in $K_{0}(\RDef_{X})$.
 \end{enumerate}
\end{thm}

\subsection{Pull-back of constructible exponential functions}  \label{inverseimage-exp}
A definable map $f:Z\ra Z'$ in $\Def_{k}$ induces a pull-back morphism (cf.  \cite[\S 5.4]{CluLoe08a}and  in \cite[\S 3.4]{CluLoe10a})
$$f^{*}:\mathcal C(Z')^{\exp} \ra \mathcal C(Z)^{\exp}.$$
Indeed, the fiber product along $f$ induces a pull-back morphism
 from $K_{0}(\RDef_{Z'})^{\exp}$ to $K_{0}(\RDef_Z)^{\exp}$ and the composition by $f$ induces also a pull-back morphism
from $\mathcal{P}(Z')$ to $\mathcal{P}(Z).$ These pull-backs are compatible with their tensor product.

\begin{rem}
 A constructible exponential function $E(g)e(\xi) \otimes \alpha \mathbb L^{\beta}$ can be thought of as
 $$z \in Z \mapsto E(g(z))e(\xi(z)) \otimes \alpha(z) \mathbb L^{\beta(z)} \in \mathcal C(\{z\})^{\exp}.$$
 More generally, the constructible exponential function $[\pi_Y] E(g)e(\xi) \otimes \alpha \mathbb L^{\beta}$ can be thought of as
 $$z \in Z \mapsto [Y_z]E(g_{\mid Y_z})e(\xi_{\mid Y_z}) \otimes \alpha_{\mid Y_z} \mathbb L^{\beta_{\mid Y_z}} \in \mathcal C(\{z\})^{\exp}.$$
 By corollary \ref{fini}, the restrictions $\alpha_{\mid Y_z}$ and $\beta_{\mid Y_z}$ take finitely many values, but $[Y_z]$ should be thought of as a kind of motive standing for a possibly infinite sum over elements in $Y_z$, which is a definable subset of some power of the residue field. With $E$ and $e$, the expression $[\pi_Y] E(g)e(\xi) $ is a kind of exponential motive, standing for possibly infinite exponential sums. In the $p$-adic case, the finiteness of the residue field allows one to see $[Y_z]$ as a finite sum again.
\end{rem}

\subsection{Push-forward of constructible exponential functions}  \label{push forward}
For $S$ in $\Def_{k}$, Cluckers and Loeser construct in
\cite{CluLoe08a} and \cite{CluLoe10a}
a functor $I_{S}^{\exp}$ from the category $\Def_{S}$ to the category $\underline{\mathrm{Ab}}$
 of abelian groups:
 $$  I_{S}^{\exp} \colon
 \begin{cases}
\Def_{S}  \longrightarrow  \underline{\mathrm{Ab}} \\
( \theta_Z : Z \to  S)  \longmapsto   (I_{S}C(\theta_Z)^{\exp} \subset C(Z)^{\exp}) \\
( \theta_Z \overset{f} \to \theta_Y)  \longmapsto  (I_{S}C(\theta_Z)^{\exp}\overset{f!}{\ra} I_{S}C(\theta_Y)^{\exp})
 \end{cases}
$$
satisfying natural axioms implying its uniqueness, see theorems 10.1.1 and 13.2.1 in \cite{CluLoe08a} and
theorem 4.1.1 in \cite{CluLoe10a}.
The elements of  $I_{S}C(\theta_Z)^{\exp}$ are called \emph{$\theta_Z$-integrable motivic constructible exponential functions} on $Z$ or simply \emph{$\theta_Z$-integrable functions}.

\begin{example} The ring $I_{S}C(\Id_S)^{\exp}$ is all of $C(S)^{\exp}$, namely, every function in $C(S)^{\exp}$ is already integrable up to $S$ itself, with the identity map $S\to S$ as structural morphism.
\end{example}

The functor $I_{S}^{\exp}$ and the integrable functions are constructed simultaneously.
The functor $I_{S}$ is first defined in \cite{CluLoe08a} in the setting without exponential and extended in \cite{CluLoe10a} in the exponential setting to $I_{S}^{\exp}$. In particular, for any $\theta_Z:Z \to S$ in
$\Def_{S}$, $I_{S}C(\theta_Z)^{\exp}$ is a graded subgroup of $C(Z)^{\exp}$ defined as
$$I_{S}C(\theta_Z)^{\exp}:=I_{S}C(\theta_Z)\otimes_{K_{0}(\RDef_Z)} K_{0}(\RDef_Z)^{\exp}.$$

\begin{rem} Sometimes we will simply say $S$-integrable instead of $\theta_Z$-integrable and write $I_{S}C(Z)^{\exp}$ when the structural morphism $\theta_Z$ is implicitly clear.
\end{rem}

The natural morphism of graded groups from $I_{S}C(\theta_Z)$ to
$I_{S}C(\theta_Z)^{\exp}$ is injective.
We will use the following axioms  (see theorem 10.1.1 in \cite{CluLoe08a} and \S 13.2 in \cite{CluLoe10a}):

\begin{axiom}[Fubini] \label{axiom-Fubini}
Let $S$ be in $\Def_{k}$. Let $f:\theta_X\ra \theta_Y$ be a definable morphism in $\Def_S$.
A constructible function $\varphi$ on $X$ is $\theta_X$-integrable if and only if $\varphi$ is $f$-integrable and
$f_{!} \varphi$ is $\theta_Y$-integrable namely:
$$\varphi \in I_{S}C(\theta_X)^{\exp} \Leftrightarrow  \varphi \in I_{Y}C(f)^{\exp}\:\:\mathrm{and}\:\: f_{!}\varphi \in I_{S}C(\theta_Y)^{\exp}.$$
\end{axiom}

\begin{axiom}[Additivity] \label{axiom:additivity}
	Let $Z$ be a definable subassignment in $\Def_S$. Assume $Z$ is the disjoint union of two definable subassignments $Z_1$ and $Z_2$. Then, for every morphism 
	$f:Z\to Y$ in $\Def_S$, the isomorphism $C(Z)\simeq C(Z_1)\oplus C(Z_2)$ induces an isomorphism $I_{S}C(Z) \simeq I_{S}C(Z_1)\oplus I_{S}C(Z_2)$ under which we have 
	$f_{!} = f_{1!} \oplus f_{2!}$.
\end{axiom}

\begin{axiom}[Projection formula]  \label{axiom-projection}
Let $S$ be in $\Def_k$. For every morphism $f$ from $\theta_Z$ to $\theta_Y$ in $\Def_S$, and every $\alpha$ in
$\mathcal C(Y)^{\exp}$ and $\beta$ in $I_{S}C(\theta_Z)^{\exp}$, if $f^{*}(\alpha)\beta$ belongs to $I_S C(\theta_Z)^{\exp}$, then $f_{!}(f^{*}(\alpha)\beta)=\alpha f_{!}(\beta)$.
\end{axiom}

\begin{axiom}[Push-forward for inclusions] \label{axiom-push-forwar-inclusion}
Let $S$ be in $\Def_k$. Let $\theta_T:T \to S$ be a definable set in $\Def_S$. Let $Z$ and $Z'$ two definable subassignments of $T$ with $Z\subset Z'$.
Let $i:Z \to Z'$ be the corresponding inclusion and $\theta_Z$ and $\theta_{Z'}$ the corresponding restriction of $\theta_T$ to $Z$ and $Z'$.
We have $\theta_Z = \theta_{Z'} \circ i$.  
Composition with $i$ induces a morphism $i_! : K_0(\RDef_{Z}) \to K_0(\RDef_{Z'})$.
The extension by zero induces a morphism $i_! : \mathcal P(Z) \to \mathcal P(Z')$.
By compatibility with the tensor product, we get a morphism 
\begin{equation}\label{push-forward-injection-cf} 
	i_! : \mathcal C(Z) \to \mathcal C(Z').
\end{equation}
For every constructible function $\varphi$ in $\mathcal C(Z)$, the class $[\varphi]$ lies in $I_SC(\theta_Z)$ if and only if $[i_!(\varphi)]$ belongs to $I_SC(\theta_{Z'})$. If this is the case then $i_!([\varphi])=[i_!(\varphi)]$.

The morphism (\ref{push-forward-injection-cf}) is also compatible with the morphism
$i_!$ from $K_0(\RDef_{Z}^{\exp})$ to $K_0(\RDef_{Z'}^{\exp})$
also defined by extension by zero. We obtain in such a way a morphism 
$i_{!} : \mathcal C(Z)^{\exp} \to \mathcal C(Z')^{\exp}$.
As $i$ sends subassignments of $Z$ to subassignments of $Z'$ of the same dimension, there are group morphisms 
$i_{!}:\mathcal C^{\leq d}(Z)^{\exp} \to \mathcal C^{\leq d}(Z')^{\exp}$ and a graded group morphisms 
$i_{!}:\mathcal C(Z)^{\exp} \to \mathcal C(Z')^{\exp}$ which restricts to a morphism 
$i_{!}:I_{S}C(\theta_Z)^{exp} \to I_{S}C(\theta_{Z'})^{exp}$. 
\end{axiom}

\begin{axiom}[Projection along $k$-variables] \label{axiom:projk}
	Let $S$ be a definable subassignment in $\Def_k$. Let $\theta_Y : Y \to S$ be in $\Def_S$. Let $n\geq 0$ be an integer. 
	We denote by $Z$ the definable set $Y[0,n,0]$, by $\pi : Z \ra Y$ the canonical projection, and by $\theta_Z$ the structural map $\pi \circ \theta_Y$.
	By proposition \ref{prop:dissociation}, there is a canonical isomorphism
	$$\mathcal C(Z) \simeq K_{0}(\RDef_{Z}) \otimes_{\mathcal P^{(0)}(Y)} \mathcal P(Y)$$
	which allows to define a ring morphism $\pi_{!}:\mathcal C(Z)\ra \mathcal C(Y)$ by sending $\sum_{i}a_i\otimes \varphi_i$ to 
	$\sum_{i}\pi_{!}(a_i) \otimes \varphi_i$ with $a_i$ in $K_{0}(\RDef_{Z})$, $\varphi_i$ in $\mathcal P(Y)$ and $\pi_{!}(a_i)$ defined in  subsection \ref{Res-pull-back-push-forward}.
	For any constructible function $\varphi$ in $\mathcal C(Z)$, the class $[\varphi]$ is $\pi$-integrable and 
	$$\pi_{!}([\varphi])=[\pi_{!}(\varphi)]$$ 
	where $\pi_{!}$ is defined above.
	
	The map $\pi_{!}:K_{0}(\RDef_Z^{\exp}) \to K_{0}(\RDef_Y^{\exp})$ induces a ring morphism 
	$$\pi_{!}:\mathcal C(Z)^{\exp} \to \mathcal C(Y)^{\exp}.$$
	Furthermore, as $\pi$ sends subassignments of $Z$ to subassignments of $Y$ of the same dimension, there are group morphisms $\pi_{!}:\mathcal C^{\leq d}(Z)^{\exp} \to \mathcal C^{\leq d}(Y)^{\exp}$ for all $d$, and a graded group morphism 
	$\pi_!:C(Z)^{\exp} \to C(Y)^{\exp}$ which restricts to a morphism
	$$\pi_!:I_{S}C(\theta_Z)^{\exp} \to I_{S}C(\theta_Y)^{\exp}.$$ 
\end{axiom}

\begin{axiom}[Projection along $\mathbb Z$-variables] \label{axiom:projZ}
Let $S$ be a definable subassignment in $\Def_k$. Let $\theta_Y : Y \to S$ be in $\Def_S$. Let $r\geq 0$ be an integer. 
	We denote by $Z$ the definable set $Y[0,0,r]$, by $\pi : Z \ra Y$ the canonical projection, and by $\theta_Z$ the structural map $\pi \circ \theta_Y$.

	Let $\varphi$ be a constructible function in $\C(Z)$. 
The class $[\varphi]$ is $\pi$-integrable if and only if $\varphi$ can be written as $\varphi = \varphi_{Y} \otimes \varphi_{\mathcal P}$, where
$\varphi_{Y}$ is a constructible function in $\C(Y)$ and $\varphi_{\mathcal P}$ is a Presburger function in $I_{Y}\mathcal P(Z)$, namely $\varphi_{\mathcal P}$ is a $Y$-integrable Presburger function on $Z$: for any $q>1$, for any $y$ in $Y$ the family
$\sum_{x\in \mathbb Z^r} \nu_q(\varphi_{\mathcal P}(y,x))$ is summable with $\nu_q : \mathbb A \to \mathbb R$ is the unique morphism of rings mapping $\mathbb L$ to $q$.
In that case we have 
$$\pi_{!}[\varphi] = [\varphi_{Y}\otimes \pi_{!}(\varphi_{\mathcal P})],$$ where 
$\pi_{!}(\varphi_{\mathcal P})$ is the unique constructible function in $\C(Y)$ such that for any $y$ in $Y$, for any $q>1$
$$\nu_q\left( (\pi_{!}\varphi_{\mathcal P})(y) \right) = \sum_{x\in \mathbb Z^r} \nu_q(\varphi_{\mathcal P}(y,x)).$$
This defines a morphism $\pi_{!}:I_{S}C(\theta_Z)\to I_{S}C(\theta_Y)$, which induces by tensor product a graded group morphism
	$$\pi_{!}:I_{S}C(\theta_Z)^{\exp} \to I_{S}C(\theta_Y)^{\exp}$$
	using the fact that the canonical morphism 
	$$K_{0}(\RDef_Y^{\exp}) \otimes_{\mathcal P^{0}(Y)} \mathcal P^{0}(Z) \ra K_{0}(\RDef_{Z}^{\exp}),$$
	is an isomorphism.
\end{axiom}

\begin{axiom}[Volume of graph; 0-cell] \label{volume-0-cell}
	Let $\theta_Y$ be in $\Def_{S}$, and
	$$Z=\{(y,z)\in Y[1,0,0] \mid z=c(y)\}$$
	where $c:Y \ra h[1,0,0]$ is a definable morphism. Denote by
	$f:Z\ra Y$ the morphism induced
	by the projection from $Y \times h[1,0,0]$ to $Y$, and $\theta_Z$ its composition with $\theta_Y$. Then, the constructible function $[1_Z]$ is $\theta_Z$-integrable if and only if $\mathbb L^{(\ord \text{jac} f)\circ f^{-1}}$ is $\theta_Y$-integrable. 
	In that case, in the ring $I_SC(Y)^{\exp}$ we have the equality $$f_{!}([1_Z]) = \mathbb L^{(\ord \text{jac} f)\circ f^{-1}}.$$
  \end{axiom}

\begin{axiom}[Volume of balls; 1-cell]  \label{volume-boule}
 Let $\theta_Y$ be in $\Def_{S}$, and
 $$Z=\{(y,z)\in Y[1,0,0] \mid \ord(z-c(y))=\alpha(y),\: \ac(z-c(y))=\xi(y)\}$$
 where $\alpha:Y\ra \mathbb Z$,
 $\xi:Y \ra h[0,1,0]\setminus \{0\}$ and
 $c:Y \ra h[1,0,0]$ are definable morphisms. Denote by
 $f:Z\ra Y$ the morphism induced
 by the projection from $Y \times h[1,0,0]$ to $Y$, and $\theta_Z$ its composition with $\theta_Y$. Then, the constructible function $[1_Z]$ is
 $\theta_Z$-integrable if and only if
 $\mathbb L^{-\alpha-1}[1_{Y}]$ is $\theta_Y$-integrable. In that case, in the ring $I_SC(Y)^{\exp}$ we have the equality $f_{!}([1_Z]) = \mathbb L^{-\alpha -1}[1_{Y}]$.
\end{axiom}

 By Axiom \ref{volume-boule}, the volume of a ball
 $\{z \in h[1,0,0] \mid \ord(z-c)=\alpha, \ac(z-c)=\xi\}$
 with $\alpha$ in $\mathbb Z$, $c$ in $k\llp t \rrp$ and $\xi$ in $k$, $\xi\neq 0$ is $\mathbb L^{\alpha-1}$.
 This is natural by analogy with the $p$-adic case.

\begin{axiom}[Relative balls of large volume]
 Let $\theta_Y$ be in $\Def_{S}$ and
 $$Z=\{(y,z)\in Y[1,0,0] \mid \ord z =\alpha(y),\: \ac z =\xi(y)\}$$
 where $\alpha:Y\ra \mathbb Z$, $\xi:Y \ra h[0,1,0]\setminus \{0\}$ are definable morphisms. Let $f:Z\ra Y$ be the morphism induced by the projection from $Y[1,0,0]$ to $Y$.
 Suppose that the constructible function $[1_{Z}]$ is $(\theta_Y \circ f)$-integrable and moreover $\alpha(y)<0$ holds for every $y$ in $Y$, then
 $f_{!}(E(z)[1_Z])=0.$
\end{axiom}

The previous axiom is also natural by analogy with the $p$-adic case, where an additive character evaluated in the identity function and integrated over a large ball is naturally zero.

  \begin{thm}[Change of variables formula, theorem 5.2.1 of \cite{CluLoe10a}]\label{cov}
   Let $f:X\ra Y$ be a definable isomorphism between definable subassignments of dimension $d$. Let $\varphi$
   be in $\mathcal{C}^{\leq d}(Y)^{\:\exp}$ with a nonzero class in $C^{d}(Y)^{\:\exp}$. Then $[f^{*}(\varphi)]$ belongs to $I_{Y}C^{d}(f)^{\exp}$ and
   $$f_{!}([f^{*}(\varphi)])=\mathbb L^{\ord (\mathrm{Jac} f) \circ f^{-1}}[\varphi].$$
  \end{thm}

    We give some ideas of the construction of this push-forward and refer to \cite{CluLoe08a} and \cite{CluLoe10a} and to the surveys \cite{CluHalLoe11}, \cite{CluLoe05a} and \cite{GorYaf09} for further details. For instance, we fix a base $S$, we consider a definable morphism $f:Y\ra S$ where $Y$ is a subassignment of some $h[m,n,r]$ and we denote by $\Gamma_f$ the graph of $f$. By functionality the morphism $f_!$ is the composition $p_! \circ i_!$ where $i : Y\ra \Gamma_f$ and $p:\Gamma_f \ra S$ are the canonical injection and projection. Thus, it is enough to know how to construct the push-forward morphisms for injections and projections. The case of definable injection is done using extension by zero of constructible functions, and an adjustment with a Jacobian to match the induced measures.
  Using the axiom of the volume of balls and the change of variables formula, we observe that the construction of the push-forward morphism for a projection is done by induction on the valued field dimension. For instance, $\Gamma_f$ can be seen as a definable subassignment of $S'[1,0,0]$ where $S'$ is the definable set $S[m-1,n,r]$ and the push-forward $p_{!}$ will be the composition $p^{(m-1)}_!\circ  \pi_{!}$ where $\pi : \Gamma_f \ra S'$ and $p^{(m-1)} : S' \ra S$ are canonical projections. The construction does not depend on the order of such projections and the main tool is the  cell decomposition theorem stated above. Once the valuative dimension is zero we have to define a push-forward of a projection from some $S[0,n',r']$ to
  $S$. This is done using the independence between the residue field and the value group, coming from theorem \ref{QE}, see for instance 
  Proposition \ref{prop:dissociation}. The push-forward along residue variables is simply the push-forward induced by composition at the level of Grothendieck ring cf. Axiom \ref{axiom:projk} (and \cite[\S 5.6]{CluLoe08a} ). The integration along $\mathbb Z$-variables corresponds to summing over the integers, cf. Axiom \ref{axiom:projZ} (and \cite[\S 4.5]{CluLoe08a}).

  \begin{example}\textbf{(Integration of constructible functions, \cite[\S 11.1]{CluLoe08a}).} \label{example-sans-exponentielle}

	  In this example, we illustrate the computation of the integration along a projection $f:S\times Y \to S$ with $Y=h[1,n,r]$.
	  We will use this computation in subsection \ref{cas-sans-exp}. 
	  Let $\varphi$ be a constructible function in $\mathcal C(S[1,n,r])$.
	  By the cell decomposition theorem (theorem \ref{celldecompthm}), there is a cell decomposition of $S[1,n,r]$ adapted to $\varphi$ denoted by $(Z_i)_{i\in I}$. For any $i$ in $I$, the cell $Z_i$ has a presentation $(\lambda_i, Z_{C_i})$, and there is a constructible function $\psi_i$ in $\mathcal C(C_i)$ such that 
	  \begin{equation} \label{equation:phii} \varphi_{\mid Z_i} = \lambda_i^* p_i^*(\psi_i)1_{Z_i} \end{equation}
	  where $Z_{C_i}$ is a subassignment of some $h[1,n+n_i,r+r_i]$, $C_i$ is a subassignment of $h[0,n+n_i,r+r_i]$, $p_i:Z_{C_i} \to C_i$ and $q_i:C_i \to S$ are the projections. By a refinement of the cell decomposition we can assume that for any $i$, the restriction $\varphi_{\mid Z_i}$ is either zero or has the same $K$-dimension as $Z_i$. By the additivity axiom (Axiom \ref{axiom:additivity}) $\varphi$ will be $f$-integrable if and only for any $i$ in $I$ the restriction $\varphi_{\mid Z_i}$ is $f$-integrable and in that case 
	  $$f_{!}\varphi = \sum_{i \in I} f_{!}\left(\varphi_{\mid Z_i}\right).$$
	  For any $i$ in $I$, we consider the commutative diagram 
	  $$\xymatrix{
		  Z_i  \ar^{\lambda_i}_{\simeq}[r] \ar_{f}[d] & Z_{C_i} \ar[d]^{p_i} \\
		  S & C_i \ar^{q_i}[l]
	  }.
	  $$
	  Using equation (\ref{equation:phii}), the projection axiom (Axiom \ref{axiom-projection}) and the Fubini axiom (Axiom \ref{axiom-Fubini}), the following statement are equivalent
	  \begin{itemize}
		  \item the restriction $\varphi_{\mid Z_i}$ is $f$-integrable,
		  \item the constructible function $p_i^{*}(\psi_i)[1_{Z_{C_i}}]$ is $f\circ \lambda_i^{-1}$-integrable 
			  namely $q_i \circ p_i$-integrable,
		  \item the constructible function $\psi_i p_{i!}[1_{Z_{C_i}}]$ is $q_i$-integrable with 
			  $p_{i!}[1_{Z{C_i}}]$ is an integration of a $0$-cell given by (Axiom \ref{volume-0-cell}) or a $1$-cell given by (Axiom \ref{volume-boule}).
	  \end{itemize}
	  The $q_i$-integrability condition and the $q_i$-integration can be treated by proposition \ref{prop:dissociation} and (Axioms \ref{axiom:projk} and \ref{axiom:projZ}). If all these constructible functions are integrable then 
	  $$f_{!}\varphi = \sum_{i \in I} q_{i!}\left(\psi_i p_{i!}[1_{Z{C_i}}]\right).$$
  \end{example}
  \begin{example}\textbf{(Integration of exponential constructible functions in \cite{CluLoe10a})} \label{example-integration-avec-exponentielle}
	  We consider two cases.
	  \begin{itemize}
		  \item[$\bullet$] Integration along the projection $\pi_S^{S[1,0,0]}:S[1,0,0] \to S$.
			  
			  Let $S$ be a definable set in $\Def_k$.
			  Let $\varphi$ be a constructible function in $\mathcal C(S[1,0,0])^{\exp}$.
			  We can write $$\varphi = [f:Y\ra S[1,0,0]]e^{\xi}E(g)\otimes \tilde{\varphi}$$
			  with $Y\subset S[1,n_Y,0]$.
			  In the construction  \cite[\S 5.1]{CluLoe10a}, using a cell decomposition adapted to $\tilde{\varphi}$ 
			  the authors consider a definable isomorphism $$\lambda : Y \to Y' \subset Y[0,n,r]$$ and the following diagram 
			  $$ 
			  \xymatrix{
				  Y \ar_{\lambda}^{\simeq}[d] \ar^{f}[rr] && S[1,0,0] \ar^{\pi_{S}}[r] & S \\
				  Y'\ar^{\pi'}[rr] && S[0,n+n_Y,r] \ar_{p'}[ur] 
			  }
			  $$
			  where $\pi'$ and $p'$ are projections . They define $\xi'=\xi \circ \lambda^{-1}$, $g'=g\circ \lambda^{-1}$ and 
			  $\varphi'=\lambda^{-1*}f^{*}(\tilde{\varphi})$.
			  By construction there exists also a unique definable function $\tilde{\xi}':S[0,n_Y+n,r] \to h[0,1,0]$ such that $\xi'=\tilde{\xi}'\circ \pi'$.

			  We consider an element of $Y'$ as a couple $(x,y)$ with $x$ in $S[0,n_Y,0]$ and $y$ in $h[1,0,0]$.
			  Following from its construction, the authors decompose the definable set $Y'$ as union $A\cup B$, such that $g'(x,.):y\to g'(x,y)$ is constant on each fiber $B_x$ and 
			  non constant and injective on each fiber $A_x$ where for each $x$, 
			  $$A_x=\{y\in h[1,0,0] \mid (x,y)\in A\}\:\: \text{and} \:\: B_x = \{y \in h[1,0,0] \mid (x,y) \in B\}.$$
			  As $g'$ is constant along fibers of $B_x$, we denote by $\tilde{g}':\pi'(B)\to h[1,0,0]$ the unique definable function such that 
			  $g'_{\mid B} = \tilde{g}' \circ \pi'_{\mid B}$.
			  They refined the above partition, decomposing $A$ as the union $A_1 \cup A_2$ with
			  $$A_1:=\{(x,y)\in A \mid g'(x,.)\:\text{maps $A_x$ onto a ball of volume $\mathbb L^{-j}$ with $j\leq 0$}\}$$
			  and 
			  $$A_2:=\{(x,y)\in A \mid g'(x,.)\:\text{maps $A_x$ onto a ball of volume $\mathbb L^{-j}$ with $j> 0$}\}$$
			  Finally by their construction they consider two definable morphisms 
			  $$r':S[0,n_Y+n,r] \to h[1,0,0] \: \text{and} \: \eta':S[0,n_Y+n,r] \to h[0,1,0]$$
			  such that for any $(x,y)$ in $A_2$
			  $$g'(x,y)-r'(x) = \eta(x) \mod (t).$$
			  Then, using all these notations they define 
			  \begin{equation} \label{formule-integration-exp-valuee}
				  \pi_{S!}^{S[0,1,0]}([\varphi]) := 
				  p'_{!}\left(e^{\tilde{\xi}'}E(\tilde{g'})\pi'_{!}([1_B\varphi'])+ e^{\tilde{\xi}'+\eta}E(r)\pi'_{!}([1_{A_2}\varphi'])\right)
		          \end{equation}

		  \item[$\bullet$] Integration along $f:Z \to Y$.  
			  Let $S$ be a definable set in $\Def_k$. Let $f:Z \to Y$ be a 
			  morphism in $\Def_S$. Let $\varphi$ in $I_SC(Z)^{\exp}$ be of the form
			  $$\varphi = E(g)e^{\eta}[p:X \ra Z]\varphi_0$$
			  with $p:X \to Z$ in $\RDef_Z$, $g:X \to h[1,0,0]$ and $\eta: \to h[0,1,0]$ definable morphisms and $\varphi_0$ in $I_SC(Z)$.
			  We denote by $\delta_{f,g,\eta}:X \to Y[1,1,0]$ the morphism sending $x$ to $\left( (f\circ p)(x),g(x),\eta(x) \right)$. 
			  We denote by 
			  $$\pi_{Y[0,1,0]}^{Y[1,1,0]} : Y[1,1,0] \to Y[0,1,0]\:\:\text{and}\:\:\pi_Y^{Y[0,1,0]}:Y[0,1,0] \to Y$$
		  the projections, and by $x$ and $\xi$ the canonical coordinates on the fibers of $\pi_{Y|0,1,0]}^{Y[1,1,0}$ and $\pi_Y^{Y[0,1,0]}$.
			  Then, using integration along a residue variable (see Axiom \ref{axiom:projk}) and integration along one valued field variable explained above, Cluckers and Loeser define 
			  $$f_{!}(\varphi) := \pi_{Y!}^{Y[0,1,0]}\left(\pi_{Y[0,1,0]!}^{Y[1,1,0]}\left(E(x)e^{\xi}\delta_{f,g,\eta!}(p^ *\varphi_0)\right)\right).$$
	  \end{itemize}
 \end{example}

 \subsection{Commutativity of pull-back and push-forward functors}

\begin{notations} Let $A$, $B$, $C$, $D$ some sets. Let $f:A\ra C$ and $g:B\ra D$ be some applications. We denote by $f\times g$ 
	the application from $A\times B$ to $C\times D$ which maps $(a,b)$ to $(f(a),g(b))$. Let $\varphi : A\times B \ra C$ and 
	$\psi:A\times B \ra D$ be some applications. We denote by $(\varphi,\psi)$ the application from $A\times B$ to $C\times D$ which maps
	$(a,b)$ to $(\varphi(a,b),\psi(a,b))$.
\end{notations}

In \cite{WF}, the second author introduced a notion of \emph{definable distributions} in Cluckers-Loeser motivic setting. He introduced also a notion of \emph{motivic wave front set} of a definable distribution, which allows him, as in the real setting \cite{Hormander83} or in the p-adic setting \cite{Hei85a} and \cite{UD}, to study the pull-back of a distribution by a function which requires the natural following commutativity relation between pull-back and push-forward functors that we prove in section \ref{proof}.

\begin{lem} 
	\label{projlem} 
	Let $S$ be a definable set in $\Def_k$ and $\gamma : W \ra W'$ be a definable morphism in $\Def_S$. Let $X$ be a definable set in $\Def_S$. We denote by $\pi_W$ the projection from $W\times_S X$ to $W$ and by $\pi_{W'}$ the projection from $W' \times_S X$ to $W'$. Let $\varphi$ be a constructible exponential function in $\C(W'\times_S X)^{exp}$.
	\begin{enumerate}
		\item \label{thm:integrability-condition} If $[\varphi]$ is $\pi_{W'}$-integrable then 
			$[(\gamma \times Id_X)^*\varphi]$ is $\pi_W$-integrable.
			Furthermore, if $\gamma$ is onto then this implication is an equivalence.
		\item \label{thm:equality} If $[\varphi]$ satisfies the condition (\ref{thm:integrability-condition}) then  
			\begin{equation}
				\pi_{W!}\left[(\gamma\times Id_X)^*\varphi \right] = \gamma^*(\pi_{W'!} [\varphi]).
			\end{equation}
	\end{enumerate}
\end{lem}

This lemma can be generalized in the following way 
\begin{thm}
\label{mainthm}
	Let $S$ be a definable set in $\Def_k$ and $\gamma : W \ra W'$ be a definable morphism in $\Def_S$. Let $f:X\to Y$ be a definable morphism in $\Def_S$. We denote by $\pi_W$ the projection from $W\times_S X$ to $W$ and by $\pi_{W'}$ the projection from $W' \times_S X$ to $W'$.
Let $\varphi$ be a constructible function in $\C(W'\times_S X)^{\exp}$.
	\begin{enumerate}
		\item \label{thm:integrability-condition-f} If $[\varphi]$ is $(\pi_{W'}\times f)$-integrable then 
			$[(\gamma \times Id_X)^*\varphi]$ is $(\pi_W \times f)$-integrable.
			Furthermore, if $\gamma$ is onto then this implication is an equivalence.
		\item \label{thm:equality-f} If $[\varphi]$ satisfies the condition (\ref{thm:integrability-condition-f}) then  
			\begin{equation}
				(\pi_{W}\times f)_{!}\left[(\gamma\times Id_X)^*\varphi \right] = 
				(\gamma\times Id_{Y})^*\left((\pi_{W'}\times f)_{!} [\varphi]\right).
			\end{equation}
	\end{enumerate}
\end{thm}

\section{Proofs} \label{proof}
In subsection \ref{preliminary-lemma}, we state and prove three lemmas allowing us to prove lemma \ref{projlem} and theorem \ref{mainthm} in an inductive way following step by step the motivic integration construction in \cite{CluLoe08a} and \cite{CluLoe10a}. 
In subsection \ref{proof-thm} we prove theorem \ref{mainthm} as a corollary of \ref{projlem}. 
In subsection \ref{cas-sans-exp} we give a proof of the projection lemma in the case of constructible functions without exponential, then in subsection \ref{exponential-case} we give the proof of the general case with exponential.

\subsection{Some preliminary lemmas} \label{preliminary-lemma}

\subsubsection{Extension lemma}

\begin{lem} \label{lem:extension}
	Let $S$ be a definable set and $\gamma : W \ra W'$ be a definable morphism in $\Def_S$. Let $X$ and $Y$ be two definable sets in $Def_S$ such that $W\times_S X$ and $W'\times_S X$ are respectively definable subassignments of $W\times_S Y$ and $W'\times_S Y$. Let $i_{W}$ and $i_{W'}$ be the corresponding canonical injections.
	For any constructible exponential function $\varphi$ in $\C(W'\times_S X)^{exp}$, $[\varphi]$ is $i_{W'}$-integrable and $[(\gamma\times Id_X)^*\varphi]$ is $i_W$-integrable. Furthermore, $\varphi$ satisfies the equality
	\begin{equation} \label{eq:extension}
		(\gamma\times Id_{Y})^{*}(i_{W'!}[\varphi]) = i_{W!}[(\gamma\times Id_X)^*\varphi].
	\end{equation}
\end{lem}
\begin{proof}
	Let $\varphi$ be an exponential constructible function in $\mathcal C(W'\times_S X)^{exp}$. By Axiom \ref{axiom-push-forwar-inclusion} (see in particular \cite[\S 5.5]{CluLoe08a} and \cite[\S 3.5]{CluLoe10a}), the classes $[\varphi]$ and $[(\gamma \times Id_X)^*\varphi]$ are respectively $i_{W'}$-integrable and $i_W$-integrable with the equalities
	$$i_{W'!}[\varphi] = [i_{W'!}(\varphi)]\:\: \text{and} \:\:i_{W!}[(\gamma \times Id_X)^*\varphi] = [i_{W!}\left((\gamma \times Id_X)^*\varphi\right)].$$
        The equality (\ref{eq:extension}) is then implied by the equality
	\begin{equation} \label{eqinclusion}
		(\gamma \times Id_Y)^*\left(i_{W'!}(\varphi)\right) = i_{W!}\left((\gamma \times Id_X)^*\varphi\right).
	\end{equation}
	which comes from the definition of the pull-back morphism in \cite[\S 5.1, \S 5.4]{CluLoe08a} and \cite[\S 3.2, \S 3.4]{CluLoe10a} 
	(see in particular subsection \ref{Res-pull-back-push-forward}) and the push-forward morphism for an injection 
	\cite[\S 5.5]{CluLoe08a} and \cite[\S 3.5]{CluLoe10a} (see also Axiom \ref{axiom-push-forwar-inclusion}). 
        Indeed, we write the exponential constructible function $\varphi$ as
	$$\varphi = [p:Z\to W'\times_S X]E(g)e(\xi)\otimes \varphi_{P}$$
	where $Z$ in $\RDef_{W'\times_S X}$, $g:Z\to h[1,0,0]$ and $\xi:Z \to h[0,1,0]$ are two definable maps and $\varphi_P$ is a Presburger function in 
	$\mathcal P(W'\times_S X)$.
	Then, the main point is that the fiber product 
	$Z \times_{(W'\times_{S} X)} (W \times_{S} X)$ of $p$ and $\gamma \times Id_X$ is isomorphic to the fiber product of 
	$i_{W'}\circ p$ and $\gamma \times Id_{Y}$. Indeed, this is a consequence from a direct computation or from the fact that $(W\times_{S} X, i_W, \gamma\times Id_X)$ is isomorphic to the fiber product of $\gamma\times Id_{Y}$ and $i_{W'}$ and the result follows from the classical pull-back lemma in the following diagram 
	$$\xymatrix{	
		Z\times_{(W'\times_S Y)} (W\times_{S}Y) \simeq Z \times_{(W'\times_S X)}(W\times_S X) \ar[rr] \ar[d] \cartesien && Z \ar[d]^{p} \\
		W\times_S X \ar[rr]^{\gamma \times Id_{X}} \ar[d]^{i_W} \cartesien && W'\times_S X \ar[d]^{i_{W'}} \\
		W\times_S Y \ar[rr]^{\gamma \times Id_{Y}} && W'\times_S Y
	}.
	$$
\end{proof}

\subsubsection{Splitting lemma}
\begin{lem} \label{lem:splitting}
	Let $S$ be a definable set and $\gamma : W \ra W'$ be a definable morphism in $\Def_S$.
	Let $X$, $Y$ and $Z$ be definable sets in $Def_S$. We consider the following commutative diagram composed with definable morphisms in $\Def_S$ 
	$$\xymatrix{
		W \times_{S} X \ar^-{\gamma \times Id_X}[r] \ar_{f}[d] & W' \times_{S} X \ar^-{f'}[d] \\
		W \times_{S} Y \ar^-{\gamma \times Id_Y}[r] \ar_{g}[d] & W' \times_{S} Y \ar^-{g'}[d] \\
		W \times_{S} Z \ar^-{\gamma \times Id_Z}[r] &  W' \times_{S} Z
	}.
	$$
	Assume 
	\begin{itemize}
		\item[$\bullet$]
			for any constructible exponential function $\varphi$ in $\C(W'\times_S X)^{exp}$, if $[\varphi]$ 
			is $f'$-integrable then $(\gamma \times Id_{X})^*[\varphi]$ is 
			$f$-integrable (with equivalence if $\gamma$ is onto) and in that case 
			\begin{equation} \label{eq1} 
				(\gamma \times Id_{Y})^*(f_{!}'[\varphi]) = 
				f_{!}[\left( \gamma\times Id_{X} \right)^*\varphi].
			\end{equation}
		\item[$\bullet$]
			for any constructible exponential function $\psi$ in $\C(W'\times_S Y)^{exp}$, if $[\psi]$ 
			is $g'$-integrable then $(\gamma \times Id_{Y})^*[\psi]$ is 
			$g$-integrable (with equivalence if $\gamma$ is onto) and in that case 
			\begin{equation} \label{eq2}
				(\gamma\times Id_Z)^*(g_{!}'[\psi]) = 
				g_{!}[\left( \gamma\times Id_{Y}\right)^*\psi].
			\end{equation}
	\end{itemize}
	then, for any constructible exponential function $\varphi$ in $\C(W'\times_S X)^{exp}$, if $[\varphi]$ is 
	$(g'\circ f')$-integrable then $(\gamma \times Id_{X})^*[\varphi]$ is 
	$(g\circ f)$-integrable (with equivalence if $\gamma$ is onto) and in that case 
	\begin{equation} \label{eq3}
		(\gamma\times Id_Z)^*((g'\circ f')_{!}[\varphi]) = 
		(g\circ f)_{!}[\left( \gamma\times Id_{X} \right)^*\varphi].
	\end{equation}
\end{lem}
\begin{proof}
	The lemma follows from Fubini axiom (see Axiom \ref{axiom-Fubini}) and the assumptions.
	
	Indeed, let $\varphi$ be a constructible exponential function in $\C(W'\times_S X)^{exp}$. By Fubini axiom, $[\varphi]$ is $(g'\circ f')$-integrable 
	if and only if $[\varphi]$ is $f'$-integrable and $f'_{!}[\varphi]$ is $g'$-integrable.
	Then, using assumptions we observe that if $[\varphi]$ is $(g'\circ f')$-integrable then $(\gamma \times Id_{X})^*[\varphi]$ is $f$-integrable and $(\gamma \times Id_{Y})^*(f'_{!}[\varphi])$ is $g$-integrable (with equivalence if $\gamma$ is onto).
	By Fubini axiom, $(\gamma \times Id_{X})^*[\varphi]$ is $f$-integrable and 
	$f_{!}(\gamma \times Id_{X})^*[\varphi]$ is $g$-integrable if and only if  $(\gamma \times Id_{X})^*[\varphi]$ is $(g \circ f)$-integrable. Then, equation (\ref{eq2}) implies the result. Furthermore, we obtain the equality (\ref{eq3}) by a direct computation using (\ref{eq1}), (\ref{eq2}) and Fubini axiom 
	$$(\gamma\times Id_Z)^*((g'\circ f')_![\varphi]) = (\gamma\times Id_Z)^*(g'_{!}(f'_{!}[\varphi])) $$
	with
	$$
	(\gamma\times Id_Z)^*(g'_{!}(f'_{!}[\varphi])) 
	= g_{!}((\gamma\times Id_{Y})^*(f'_![\varphi]))
	= g_{!}(f_{!}[(\gamma\times Id_{X})^*\varphi]),$$
	then 
	$$(\gamma\times Id_Z)^*((g'\circ f')_![\varphi]) =
	(g \circ f)_!([(\gamma \times Id_{X})^*\varphi]).
	$$

\end{proof}

\subsubsection{Reduction lemma}

\begin{rem}  \label{rem:produit} Let $S$ be a definable set and $W$ be a definable set in $\Def_S$. Let $m$, $n$ and $r$ be some non negative integers. The fiber product $W\times_S S[m,n,r]$ is isomorphic to $W[m,n,r]$ and we identify them in the following.
\end{rem}

\begin{lem} \label{lem:reduction}
	Let $S$ be a definable set and $\gamma : W \ra W'$ be a definable morphism in $\Def_S$. Let $m$, $n$ and $r$ be some non negative integers. Assume lemma \ref{projlem} true in the $\mathcal C(W'[m,n,r])^{exp}$ case then it is true in the
	$\mathcal C(W'\times_S X)^{exp}$ case for any definable subassignment $X$ of $S[m,n,r]$.
\end{lem}
\begin{proof} 
	Considering the assumption and the diagram
$$\xymatrix{
	W\times_S X \ar[d]_{i} \ar[rr]^{\gamma \times Id_X} & & W' \times_S X \ar[d]^{i'} \\
	W[m,n,r] \ar[rr]^{\gamma \times Id_{S[m,n,r]}} \ar[d]_{\pi_W}  & & W'[m,n,r] \ar[d]^{\pi_{W'}} \\
	W \ar[rr]^{\gamma} & & W'
}
$$
the lemma follows from the extension lemma \ref{lem:extension} and the splitting lemma \ref{lem:splitting}. 
\end{proof}

\subsection{Proof of theorem \ref{mainthm}} \label{proof-thm}
In this subsection, we assume lemma \ref{projlem} true and we prove theorem \ref{mainthm} as a consequence of the extension lemma \ref{lem:extension} and the splitting lemma \ref{lem:splitting}.

\begin{proof}
Let $S$ be a definable set in $\Def_k$ and $\gamma : W \ra W'$ be a definable morphism in $\Def_S$. Let $f:X\to Y$ be a definable morphism in $\Def_S$. We denote by $\Gamma_f$ the graph of $f$, by $i_f:X \to \Gamma_f$ the canonical injection.
In the following we will identify canonically $W\times_S \Gamma_f$ and 
$W'\times_S \Gamma_f$ with $(W\times_S Y) \times_Y X$ and $(W'\times_S Y) \times_Y X$. We consider the following commutative diagram 

$$\xymatrix{
	& W\times_S X \ar[dddd]_{Id_W\times f} \ar[ld]_{Id_W \times i_f} \ar[r]^{\gamma \times Id_X} & W'\times_S X \ar[dr]^{Id_W' \times i_f} \ar[dddd]^{Id_{W'}\times f} &    \\
	W\times_S \Gamma_f  \ar[rrr]^{\gamma \times Id_{\Gamma_f}} &                          &                   &    W'\times_S \Gamma_f  \\ 
        \simeq & & & \simeq \\
	(W\times_S Y) \times_Y X  \ar[dr]_{\pi_{W\times_S Y}} \ar[rrr]^{\gamma \times Id_Y \times Id_X} & & & (W'\times_S Y) \times_Y X \ar[dl]^{\pi_{W'\times_{S} Y}}\\
	& W\times_S Y \ar[r]^{\gamma \times Id_Y}   & W'\times_S Y  & 
}
$$
The theorem follows from this diagram and the splitting lemma \ref{lem:splitting} (or very similar arguments) whose assumptions are satisfied by application of \begin{itemize}
	\item[$\bullet$]  the extension lemma \ref{lem:extension} for the diagram 
        $$\xymatrix{
	 W\times_S X \ar[d]_{Id_W \times i_f} \ar[r]^{\gamma \times Id_X} & W'\times_S X \ar[d]^{Id_W' \times i_f}  \\
         W\times_S \Gamma_f  \ar[r]^{\gamma \times Id_{\Gamma_f}} &  W'\times_S \Gamma_f
          }
	$$

\item[$\bullet$] the projection lemma \ref{projlem} (relatively to $Y$) for the diagram 
$$\xymatrix{ 
	(W\times_S Y) \times_Y X  \ar[d]_{\pi_{W\times_S Y}} \ar[rr]^{\gamma \times Id_Y \times Id_X} & & (W'\times_S Y) \times_Y X \ar[d]^{\pi_{W'\times_S Y}}\\
	 W\times_S Y \ar[rr]^{\gamma \times Id_Y}  &  & W'\times_S Y   
}.
$$
\end{itemize}
\end{proof}

\subsection{Proof of the projection lemma for constructible functions without exponentials} \label{cas-sans-exp}

We prove in this subsection the projection lemma \ref{projlem} for constructible functions in $C(W'\times_S X)$-case. The exponential case will be proved in subsection \ref{exponential-case}.

\subsubsection{Case $X=S[0,0,r]$} \label{case:S[0,0,r]}
In this subsection, we prove lemma \ref{projlem} in the case $X=S[0,0,r]$. We use remark \ref{rem:produit} and notations of lemma \ref{projlem}. 
Let $\varphi$ be a constructible function in $\C(W'\times_S X)$. 
By proposition \ref{prop:dissociation} we write $\varphi = \varphi_{W'} \otimes \varphi_{\mathcal P}$, where $\varphi_{W'}$ is a constructible function in $\mathcal C(W')$ and 
$\varphi_P$ is a Presburger function in $\mathcal P(W'\times_S X)$. By Axiom \ref{axiom:projZ}, $[\varphi]$ is $\pi_{W'}$-integrable if and only if $\varphi_P$ is $W'$-integrable. The pull-back $(\gamma\times Id_{X})^*\varphi$ is equal to 
$(\varphi_{W'}\circ \gamma) \otimes (\varphi_{\mathcal P} \circ (\gamma \times Id_{X}))$, then by Axiom \ref{axiom:projZ}, we deduce that 
if $\varphi$ is $\pi_{W'}$-integrable then $(\gamma\times Id_{X})^*\varphi$ is $\pi_W$-integrable and furthermore if $\gamma$ is onto then this is an equivalence. 
Under the integrability assumption, $\pi_{W'!}[\varphi]$ is equal to the class of $\varphi_{W'}\otimes \pi_{W'!}(\varphi_{\mathcal P})$, where 
$\pi_{W'!}(\varphi_{\mathcal P})$ is the unique Presburger function in $\mathcal P(W')$ such that for any $w'$ in $W'$, for any $q>1$
$$\nu_q\left( (\pi_{W'!}\varphi_{\mathcal P})(w') \right) = \sum_{x\in \mathbb Z^r} \nu_q(\varphi_{\mathcal P}(w',x)).$$
In particular, 
$$ \gamma^{*}\left(\pi_{W'!}[\varphi]\right) = [(\varphi_{W'}\circ \gamma) \otimes (\pi_{W'!}(\varphi_{\mathcal P})\circ \gamma)].$$
As well, 
$\pi_{W!}\left[ (\gamma\times Id_X)^*\varphi \right]$ is equal to the class 
$[(\varphi_{W'}\circ \gamma) \otimes \pi_{W!}(\varphi_{\mathcal P} \circ (\gamma \times Id_{X}))]$ where 
$\pi_{W!}(\varphi_{\mathcal P} \circ (\gamma \times Id_{X}))$ is the unique Presburger function in $\mathcal P(W)$ such that 
for any $w$ in $W$, for any $q>1$ 
$$\nu_q\left(\pi_{W!}(\varphi_{\mathcal P} \circ (\gamma \times Id_{X}))(w)\right) =
\sum_{x \in \mathbb Z^r} \nu_q\left( (\varphi_{\mathcal P} \circ (\gamma \times Id_{X}))(w,x)
\right).
$$
But, for any $q>1$ and $w$ in $W$ 
$$\nu_q\left(\pi_{W!}(\varphi_{\mathcal P} \circ (\gamma \times Id_{X}))(w)\right) = 
\sum_{x \in \mathbb Z^r} \nu_q\left( \varphi_{\mathcal P}(\gamma(w),x) \right) = \nu_q \left( \pi_{W'!}\varphi_{\mathcal P}(\gamma(\omega)) \right)
$$
then by uniqueness $\pi_{W!}(\varphi_{\mathcal P} \circ (\gamma \times Id_{X}))$ is equal to 
$\pi_{W'!}(\varphi_{\mathcal P}) \circ \gamma$ and we deduce the equality \ref{thm:equality} of lemma \ref{projlem}.

\subsubsection{Case $X=S[0,n,0]$} \label{case:S[0,n,0]}
In this subsection, we prove lemma \ref{projlem} in the case where $X=S[0,n,0]$. We use remark \ref{rem:produit} and notations of lemma \ref{projlem}. Let $\varphi$ be a constructible function in $\C(W'\times X)$. 
By proposition \ref{prop:dissociation}, $\varphi$ can be written as 
$\varphi = [p:Y\ra W'\times X] \otimes \varphi_{\mathcal P}$ with 
$[p:Y\ra W'\times X]$ in $K_0(\RDef_{W'\times X})$ and $\varphi_{\mathcal P}$ in $\mathcal P(W')$. By Axiom \ref{axiom:projk}, 
the class $[\varphi]$ is $\pi_{W'}$-integrable with 
$$\pi_{W'!}[\varphi] = [\pi_{W'}\circ p : Y \ra W'] \otimes \varphi_{\mathcal P}$$ and 
$\gamma^*(\pi_{W'!}[\varphi]) = [Y\times_{W'}W \ra W]\otimes (\varphi_{\mathcal P} \circ \gamma)$.
As well, $(\gamma \times Id_X)^*\varphi$ is equal to 
$[Y\times_{W'\times X}(W\times X)\ra W\times X]\otimes (\varphi_{\mathcal P} \circ \gamma)$ 
and is $\pi_W$-integrable. As $(W\times X, \pi_W, \gamma\times Id_X)$ is isomorphic to the fiber product $W\times_{W'} (W'\times X)$ of $\gamma$ and $\pi_{W'}$, we deduce similarly to the case $X=S[0,0,r]$, by the classical pull-back theorem, that $Y\times_{W'\times X}(W\times X)$ and $Y\times_{W'} (W\times X)$ are isomorphic, which induces the equality \ref{thm:equality} of lemma \ref{projlem}.

\subsubsection{Case $X=S[0,n,r]$} \label{case:S[0,n,r]}
The lemma \ref{projlem} in the case $X=S[0,n,r]$ follows immediately from the splitting lemma \ref{lem:splitting} applied to the case $X=S[0,0,r]$ in \ref{case:S[0,0,r]} and the case $X=S[0,n,0]$ in \ref{case:S[0,n,0]}.

\subsubsection{Case where $X$ is a $0$-cell and $\varphi=1_{W'\times_S X}$} \label{case:0-cell}
Let $X$ be a $0$-cell of base $Y$ in $\Def_S$ with center $c:Y \ra h[1,0,0]$ 
$$X=\{(y,z)\in Y[1,0,0] \mid z=c(y)\}.$$
We denote by $\pi$ the projection from $X$ to $Y$ which is an isomorphism. 
The product $W\times_S X$ is also a $0$-cell of base $W\times_S Y$ and center $c_W = c \circ \pi_{Y}$ where $\pi_{Y}$ is the projection from $W\times_S Y$ to $Y$. As well the product $W'\times_S X$ is a $0$-cell of base $W'\times_S Y$, center $c_{W'} = c \circ \pi'_{Y}$ where $\pi'_{Y}$ is the projection from $W'\times_S Y$ to $Y$.
We prove lemma \ref{projlem} for the constructible function $\varphi=1_{W'\times_S X}$ and the diagram
$$\xymatrix{
	W\times_S X \ar^{\gamma \times Id_X}[r] \ar_{p_W=Id_W\times \pi}[d] & W'\times_S X \ar^{p_{W'}=Id_{W'}\times \pi}[d]\\
	W\times_S Y \ar_{\gamma \times Id_Y}[r] & W'\times_S Y
}.
$$
By definition of the pull-back of a Presburger function we have
$$(\gamma \times Id_X)^*\varphi = \varphi \circ (\gamma \times Id_X) = 1_{W\times_S X}.$$
Then, by the change variable formula \cite[Proposition 13.2.1]{CluLoe08a} (see theorem \ref{cov}), the class $[\varphi]$ is $p_{W'}$-integrable and the class $[(\gamma\times Id_X)^{*}\varphi]$ is also $p_W$-integrable with 
$$p_{W'!}[\varphi] = \mathbb L^{\ord \Jac p_{W'} \circ p_{W'}^{-1}}\:\: \text{and}\:\:
p_{W!}[(\gamma \times Id_X)^*\varphi] = \mathbb L^{\ord \Jac p_{W} \circ p_{W}^{-1}}.$$

Remark that by definition of $p_W$ and $p_{W'}$, the order $\ord \Jac p_{W'} \circ p_{W'}^{-1}$ is equal to the order
$\ord  \Jac \pi \circ (\pi^{-1} \circ \pi'_Y)$. Similarly, the order 
$\ord \Jac p_{W} \circ p_{W}^{-1}$ is equal to the order $\ord  \Jac \pi \circ (\pi^{-1} \circ \pi_{Y})$.
Then, the equality \ref{thm:equality} of lemma \ref{projlem} follows from the equality 
$\pi_Y = \pi'_Y \circ (\gamma \times Id_Y)$.

\subsubsection{Case where $X$ is $1$-cell and $\varphi=1_{W'\times_S X}$} \label{case:1-cell}
Let $X$ be a $1$-cell of base $Y$ in $\Def_S$ with center $c: Y \ra h[1,0,0]$ and data $\alpha:Y \ra \mathbb Z$ and $\xi:Y\ra h[0,1,0]$ 
$$X=\{(y,z)\in Y[1,0,0] \mid \ord (z-c(y))=\alpha(y),\: \ac(z-c(y))=\xi(y) \}.$$
We denote by $\pi$ the projection from $X$ to $Y$.
The product $W\times_S X$ is still a $1$-cell with base $W\times_S Y$, center  $c_W=c\circ \pi_Y$ and data $\alpha_W=\alpha\circ \pi_Y$ and
$\xi_W = \xi \circ \pi_Y$, with $\pi_Y$ the projection from $W\times_S Y$ to $Y$. As well, the product $W'\times_S X$ is still a $1$-cell with base $W'\times_S Y$, center  $c_{W'}=c\circ \pi'_Y$ and data $\alpha_{W'}=\alpha\circ \pi'_Y$ and $\xi_{W'} = \xi \circ \pi'_Y$ with $\pi'_Y$ the projection from $W'\times_S Y$ to $Y$. 
We prove lemma \ref{projlem} for the constructible function $\varphi=1_{W'\times_S X}$ and the diagram
$$\xymatrix{
	W\times_S X \ar^{\gamma \times Id_X}[r] \ar_{p_W=Id_W\times \pi}[d] & W'\times_S X \ar^{p_{W'}=Id_{W'}\times \pi}[d]\\
	W\times_S Y \ar_{\gamma \times Id_Y}[r] & W'\times_S Y
}.
$$
By definition of the pull-back of a Presburger function we have
$$(\gamma \times Id_X)^*\varphi = \varphi \circ (\gamma \times Id_X) = 1_{W\times_S X}.$$
Then, by Axiom \ref{volume-boule}, (see axiom (A7) of \cite[theorem 10.1.1, Proposition 13.2.1]{CluLoe08a}), the class $[\varphi]$ is $p_{W'}$-integrable and the class $[(\gamma\times Id_X)^{*}\varphi]$ is also $p_W$-integrable with 
$$p_{W'!}[\varphi] = \mathbb L^{-\alpha_{W'}-1}[1_{W'\times_S Y}] \:\: \text{and}\:\:
p_{W!}[(\gamma \times Id_X)^*\varphi] = \mathbb L^{-\alpha_{W}-1}[1_{W\times_S Y}].$$
Then, the equality \ref{thm:equality} of lemma \ref{projlem} follows from the equality $\pi_Y = \pi'_Y \circ (\gamma \times Id_Y)$.

\subsubsection{Proof of the projection lemma \ref{projlem}} \label{cas-sans-exp-proof}
Using the reduction lemma \ref{lem:reduction}, it is enough to consider the case where $X$ is equal to the definable set $S[m,n,r]$ for $m$, $n$ and $r$ some non-negative integers. 
We use remark \ref{rem:produit} and notations of lemma \ref{projlem}. The projection lemma \ref{projlem} is proved by induction on $m$. 
The base case $m=0$ is ever considered in paragraph \ref{case:S[0,n,r]} and using the splitting lemma \ref{lem:splitting}, it is enough to prove the projection lemma \ref{projlem} for the diagram
\begin{equation} \label{keydiag}
\xymatrix{
	W\times_S X \ar^{\gamma \times Id_X}[r] \ar_{\pi_{\overline{W}}}[d] & W'\times_S X \ar^{\pi_{\overline{W'}}}[d] \\
	\overline{W} \ar_{\gamma \times Id_{S[m-1,n,r]}}[r] & \overline{W'}
}
\end{equation}
with $m\geq 1$, $\overline{W'}=W'[m-1,n,r]$, $\overline{W}=W[m-1,n,r]$ and where $\pi_{\overline{W}}$ and $\pi_{\overline{W'}}$ are the canonical projections.

We prove now the case of diagram (\ref{keydiag}) using the cell decomposition theorem \ref{celldecompthm} and the specific cases of $0$-cell and $1$-cell in paragraphs \ref{case:0-cell} and \ref{case:1-cell}. 
As $m\geq 1$, we consider $W'\times_S X$ as the product $\overline{W'}[1,0,0]$ and $W\times_S X$ as the product $\overline{W}[1,0,0]$.
Let $\varphi$ be a constructible function in $\C(W'\times_S X)$.
By theorem \ref{celldecompthm}, we consider a cell decomposition $\left((W'\times_S X)_i\right)_{i\in I}$ of $W'\times_S X$ adapted to $\varphi$, with for any $i$ in $I$, a presentation $(Z_{C_i'},\lambda_i')$ of the cell $(W'\times_S X)_i$

\begin{equation} \label{diag-presentation}
\xymatrix{
	(W'\times_S X)_i \ar^-{\lambda_i'}[r] \ar_{\pi_{\overline{W'}}}[d] & Z_{C_i'} \subset 
	\overline{W'}[1,n_i,r_i] \ar^-{p_i'}[d] \\
	\overline{W'} & C_i'\subset{\overline{W'}[0,n_i,r_i]} \ar_-{\pi_i'}[l]
}
\end{equation}
where the diagram is commutative, $p_i'$ and $\pi_i'$ are the canonical projections and  
\begin{itemize}
	\item   if $(W'\times_S X)_i$ is a $0$-cell, then the integer $r_i$ is equal to $0$ and the isomorphism $\lambda_i'$ is equal to 
		$Id_{(W'\times_S X)_i}\times \eta_i$ where $\eta_i : (W'\times_S X)_i\ra h[0,n_i,0]$ is a definable morphism. The jacobian order of the isomorphism $\lambda_i'$ is 0.
	\item   if $(W'\times_S X)_i$ is a $1$-cell, then the isomorphism $\lambda_i'$ is equal to 
		$Id_{(W'\times_S X)_i}\times \eta_i \times l_i$ where $\eta_i : (W'\times_S X)_i\ra h[0,n_i,0]$ and $l_i:(W'\times_S X)_i \ra h[0,0,r_i]$ are
		two definable morphisms. The jacobian order of the isomorphism $\lambda_i'$ is 0.

\end{itemize}

Taking a refinement of the cell decomposition, we may assume that for any $i$ in $I$, 
$\varphi_{|(W'\times_S X)_i}$ is either zero or has the same $K$-dimension as $(W'\times_S X)_i$ (see \ref{def:cef} and also proof 
\cite[\S 11.1]{CluLoe08a}). Furthermore, for any $i$ in $I$, there is a constructible function $\psi_i'$ in $\C(C_i')$ such that 
\begin{equation} \label{presentation-phi}
\varphi_{|(W'\times_S X)_i} = (\lambda_i')^*(p_i')^*\psi_i'.
\end{equation}

Taking the pull-back by the definable morphism $\gamma\times Id_X$, we deduce a cell decomposition $((W\times_S X)_i)_{i \in I}$ of $W\times_S X$ adapted to 
$(\gamma \times Id_X)^*\varphi$. 
More precisely, for any $i$ in $I$, we define the cell $(W\times_S X)_i$ as the definable set $(\gamma\times Id_X)^{-1}(W'\times_S X)_i$. This cell has a presentation $(Z_{C_i},\lambda_i)$ where
$$C_i=\{(w,x,\eta,l)\in W[m-1,n+n_i,r+r_i] \mid (\gamma(\omega),x,\eta,l) \in C_{i}'\},$$
and
\begin{itemize}
	\item  if $Z_{C_i'}$ is a $1$-cell with center $c_i'$ and data $\alpha_i'$ and $\xi_i'$, then 
		$Z_{C_i}$ is the $1$-cell in $\overline{W}[1,n_i,r_i]$ with center 
		$c_i=c_i'\circ (\gamma\times Id_{S[m-1,n+n_i,r_i]})$ and data 
		$$\alpha_i = \alpha_i'\circ (\gamma\times Id_{S[m-1,n+n_i,r_i]}) \:\: \text{and} \:\: 
		\xi_i = \xi_i' \circ (\gamma\times Id_{S[m-1,n+n_i,r_i]})$$ 
		and the presentation morphism $\lambda_i : (W\times_S X)_i \to Z_{C_i}$ is 
$$\lambda_i=Id_{(W\times_S X)_i}\times (\eta_i \circ (\gamma \times Id_X)) \times (l_i\circ (\gamma \times Id_X)).$$
	\item  if $Z_{C_i'}$ is a $0$-cell with center $c_i'$, then 
		$Z_{C_i}$ is the $0$-cell in $\overline{W}[1,n_i,0]$ with center 
		$$c_i=c_i'\circ (\gamma\times Id_{S[m-1,n+n_i,r_i]})$$ 
		and the presentation morphism $\lambda_i : (W\times_S X)_i \to Z_{C_i}$ is
$$\lambda_i=Id_{(W\times_S X)_i}\times (\eta_i \circ (\gamma \times Id_X)).$$
\end{itemize}

For any $i$ in $I$, we consider the constructible function 
$$\psi_i:=(\gamma \times Id_{S[m-1,n+n_i,r+r_i]})^*(\psi_i')$$ and we have the equalities
$$(\gamma \times Id_X)^* 1_{(W'\times_S X)_i} = 1_{(W\times_S X)_i}$$
\begin{equation} \label{presentation-pull-back-phi}
	(\gamma \times Id_X)^*\varphi \cdot 1_{(W\times_S X)_i} = 
 (\gamma \times Id_X)^* (\lambda_i'^*p_i'^* \psi_i') = \lambda_i^* p_i^* \psi_i
 \end{equation}
 thanks to the equality $$p_i'\circ \lambda_i' \circ (\gamma \times Id_X) = (\gamma \times Id_{S[m-1,n+n_i,r+r_i]})\circ p_i \circ \lambda_i$$
 with $p_i$ the canonical projection from $\overline{W}[1,n_i,r_i]$ to $\overline{W}[0,n_i,r_i]$.

 By the additivity axiom (Axiom \ref{axiom:additivity}), as the cells $(W'\times_S X)_i$ and $(W\times_S X)_i$ are disjoint:
 \begin{itemize}
	 \item  the class $[\varphi]$ is $\pi_{\overline{W'}}$\:-\:integrable if and only if for any $i$ in $I$, 
		 $\varphi 1_{(W'\times_S X)_i}$ is $\pi_{\overline{W'}}$\:-\:integrable, if and only if for any $i$ in $I$, the class $[\psi_i'] p_{i!}' [1_{Z_{C_i'}}]$ is $\pi_i'$\:-\:integrable (applying to equation \ref{presentation-phi} and diagram \ref{diag-presentation}, Fubini axiom (Axiom \ref{axiom-Fubini}), change variable formula (theorem \ref{cov}) and projection axiom (Axiom \ref{axiom-projection})), and in that case 
		 \begin{equation} \label{eq:1}
			 \pi_{\overline{W'}!}[\varphi.1_{(W'\times_S X)_i}] = \pi_{i!}'([\psi_i']p_{i!}'[1_{Z_{C_i'}}]).
		 \end{equation}
	 \item the class $[(\gamma \times Id_X)^*\varphi]$ is $\pi_{\overline{W}}$\:-\:integrable if and only if for all $i$
		 in $I$, the class of $(\gamma \times Id_X)^*\varphi\cdot 1_{(W\times_S X)_i}$ is $\pi_{\overline{W}}$\:-\:integrable,
		 if and only if for any $i$ in $I$, $[\psi_i] p_{i!} [1_{Z_{C_i}}]$ is $\pi_i$\:-\:integrable (by Fubini axiom (Axiom \ref{axiom-Fubini}), change variable formula and projection axiom (Axiom \ref{axiom-projection}) applied to equation \ref{presentation-pull-back-phi}), and in that case 
		 \begin{equation} \label{eq:2}
		   \pi_{\overline{W}!}[(\gamma \times Id_X)^*\varphi \cdot 1_{(W\times_S X)_i}] = \pi_{i!}([\psi_i]p_{i!}[1_{Z_{C_i}}]).
                 \end{equation}
 \end{itemize}
 But, by construction, $1_{Z_{C_i}} = (\gamma \times Id_{S[m,n+n_i,r+r_i]})^* 1_{Z_{C_i'}}$ and by the $0$-cell and $1$-cell case in sections $\ref{case:0-cell}$ and $\ref{case:1-cell}$ we have the identity 
 $$(\gamma \times Id_{S[m-1,n+n_i,r+r_i]})^*p_{i!}'[1_{Z_{C_i}}] = p_{i!}[1_{Z_{C_i}}]$$
 which implies the equality 
 \begin{equation} \label{eq:3}
	 [\psi_i] p_{i!}[1_{Z_{C_i}}] = (\gamma \times Id_{S[m-1,n+n_i,r+r_i]})^*([\psi_i']p_{i!}'[1_{Z_{C_i'}}]).
 \end{equation}

 Then, for any $i$ in $I$, using subsection \ref{case:S[0,n,r]} for the diagram 
 $$
 \xymatrix{ C_i \subset \overline{W}[0,n_i,r_i] \ar^{\gamma \times Id_{S[m-1,n+n_i,r+r_i]}}[rrr] \ar_{\pi_i}[d] &&&
 C_i' \subset \overline{W'}[0,n_i,r_i] \ar^{\pi_i'}[d] \\
 \overline{W} \ar^{\gamma \times Id_{S[m-1,n,r]}}[rrr] &&& \overline{W'}
 }
 $$
 with $\pi_i$ and $\pi_i'$ the canonical projections, we deduce that if the class $[\psi_i']p_{i!}'[1_{Z_{C_i}}]$ is $\pi_{i'}$-integrable then 
 $(\gamma \times Id_{S[m-1,n+n_i,r+r_i]})^*([\psi_i']p_{i!}'[1_{Z_{C_i'}}])$ equal to 
 $[\psi_i] p_{i!}[1_{Z_{C_i}}]$ is also $\pi_i$-integrable and if $\gamma$ is onto, this is an equivalence. 
 In that case we have the equality
  $$\pi_{i!}((\gamma \times Id_{S[m-1,n+n_i,r+r_i]})^*([\psi_i']p_{i!}'[1_{Z_{C_i'}}])$$
 \begin{equation} \label{eq:4}
	   = 
 \end{equation}
 $$(\gamma \times Id_{S[m-1,n,r]})^*(\pi_{i'!}([\psi_i']p_{i!}'[1_{Z_{C_i}}])).$$

 Then, we can conclude that for any $i$ in $I$, if the class 
 $[\varphi. 1_{(W'\times_S X)_i}]$ is $\pi_{\overline{W'}}$-integrable then the class 
 $[(\gamma \times Id_X)^*\varphi.1_{(W\times_S X)_i}]$ is $\pi_{\overline{W}}$-integrable and this is an equivalence in the case of $\gamma$ onto.
 In that case by equations (\ref{eq:1}), (\ref{eq:2}), (\ref{eq:3}) and (\ref{eq:4}), we get for any $i$ in $I$
 $$\pi_{\overline{W}!}[(\gamma \times Id_X)^*(\varphi.1_{(W'\times_S X)_i})] = (\gamma \times Id_{S[m-1,n,r]})^*
 \pi_{W'!}([\varphi. 1_{(W'\times_S X)_i}])$$
 By additivity Axiom \ref{axiom:additivity} and summation we conclude that if $[\varphi]$ is $\pi_{\overline{W'}}$-integrable then the class 
 $[(\gamma \times Id_X)^*\varphi]$ is $\pi_{\overline{W}}$-integrable and this is an equivalence in the case of $\gamma$ onto, and in that case 
 $$\pi_{\overline{W}!}[(\gamma \times Id_X)^*\varphi] = (\gamma \times Id_{S[m-1,n,r]})^*
 \pi_{W'!}([\varphi]).$$

 \begin{rem} \label{remsansexp} This achieves the proof of lemma \ref{projlem} in the $C(W\times_S X)$-case, which implies that theorem \ref{mainthm} is also true in this setting. We will use both of them in the proof of lemma \ref{projlem} in the $C(W\times_S X)^{exp}$-context.
 \end{rem}
 
\subsection{$C(W'\times_S X)^{exp}$-case}. \label{exponential-case}

Let $X$ be a definable in $Def_S$, by the reduction lemma \ref{lem:reduction} we can assume that $X=S[m,n,r]$.
Let $\gamma : W\to W'$ be a morphism in $\Def_S$. Let $\pi_{W'}:W'\times_S X \to W'$ and $\pi_{W}:W\times_S X \to X$ be the canonical projections.
We consider the diagram 
$$\xymatrix{
	W\times_S X \ar_{\pi_W}[d] \ar^{\gamma \times Id_X}[r] & W'\times_S X \ar^{\pi_{W'}}[d] \\
	W \ar^{\gamma}[r] & W'
}
$$
Let $\varphi$ be a constructible function in $\mathcal C(W'\times_S X)^{exp}$, then by definition $\varphi$ can be written as 
\begin{equation} \label{eq:varphi}
	\varphi = E(g')e(\eta')[p':Y'\to W'\times_S X]\otimes \varphi_0
\end{equation} 
where $g':W'\times_S X \to h[1,0,0]$, $\eta':W'\times_S X \to h[0,1,0]$, $p':Y'\to W'\times_S X$ are definable functions with $Y'$ a definable subset of some 
$(W'\times_S X)[0,n,0]$ and $\varphi_0$ is a constructible function in 
$\mathcal C(W'\times_S X)$. 
In particular the pull-back by $\gamma\times Id_X$ of $\varphi$ is 
\begin{equation}
	(\gamma\times Id_X)^* \varphi = E(g)e(\eta)[p:Y\to W\times_S X] \otimes (\gamma\times Id_X)^* \varphi_0
\end{equation}
where $p:Y\to W \times X$, $g$ and $\eta$ are the pull-back of $p'$, $g'$ and $\eta'$ by $\gamma \times Id_X$.

Furthermore, by definition (see subsection \ref{push forward}), $[\varphi]$ is $\pi_{W'}$-integrable if and only if $[\varphi_0]$ is $\pi_{W'}$-integrable and in that case, it follows from theorem 4.1.1 in \cite{CluLoe10a} and the uniqueness part of its proof $\S 6.3$ that
\begin{equation} \label{eq:int-varphi}
	\pi_{W'!}[\varphi] = 
	\left(\pi_{W'!}^{W'[0,1,0]}\right)_{!} \left(\pi_{W'[0,1,0]}^{W'[1,1,0]}\right)_{!}\left(E(x)e(\xi)\otimes \delta'_{!}(p'^{*}[\varphi_0])\right).
\end{equation}
where we use the diagram
$$
\xymatrix{
	Y' \ar_{p'}[dd] \ar^{\delta'}[rr] & & W'[1,1,0] \ar^{\pi^{W'[1,1,0]}_{W'[0,1,0]}}[d] \\
	   & & W'[0,1,0] \ar^{\pi^{W'[0,1,0]}_{W'}}[d]\\
	   W\times_S X \ar_{\pi_{W'}}[rr]  & &  W'
}
$$
where $x$ and $\xi$ are coordinate on the vector bundle $W'[1,1,0]$ and $\delta'$ is the definable function from $Y'$ to $W'[1,1,0]$ equal to $(\pi_{W'}\circ p',g'\circ p', \eta'\circ p')$. 
In particular, we deduce from the case without exponential in section \ref{cas-sans-exp}, that condition \ref{thm:integrability-condition} of lemma \ref{projlem} is satisfied. 
We have the following commutative diagram

$$\xymatrix{
	W[1,1,0] \ar^{\gamma \times Id_{S[1,1,0]}}[rrr] \ar_{\pi_{W[0,1,0]}^{W[1,1,0]}}[ddd] &           &            & W'[1,1,0] \ar^{\pi_{W'[0,1,0]}^{W'[1,1,0]}}[ddd] \\
	& Y \ar^{\gamma \times Id_{X[0,n,0]}}[r] \ar_{p}[d] \ar^{\delta}[lu]       & Y' \ar^{p'}[d] \ar_{\delta'}[ru]        &           \\ 
	& W\times X \ar^{\gamma \times Id_{X}}[r] \ar^{\pi_W}[lddd] & W'\times X \ar_{\pi_{W'}}[rddd]  &           \\
	W[0,1,0] \ar^{\gamma \times Id_{S[0,1,0]}}[rrr] \ar_{\pi_{W}^{W[0,1,0]}}[dd] &           &            & W'[0,1,0] \ar^{\pi_{W'}^{W'[0,1,0]}}[dd] \\ \\
	W  \ar^{\gamma}[rrr]      &           &            & W'
}.
$$
Using the splitting lemma (lemma \ref{lem:splitting}) to prove the equality \ref{thm:equality} of lemma \ref{projlem}

$$\gamma^{*}(\pi_{W'!}[\varphi]) = \pi_{W!}([(\gamma\times Id_X)^*\varphi])$$
it is enough to prove the result for the special cases 

\begin{equation}\label{cas1}
	\xymatrix{ W[0,1,0] \ar_{\pi_{W}^{W[0,1,0]}}[d] \ar^{\gamma \times Id_{S[0,1,0]}}[rr] && W'[0,1,0] \ar^{\pi_{W'}^{W'[0,1,0]}}[d]  \\
	W \ar_{\gamma}[rr] && W'}
\end{equation}
        
\begin{equation}\label{cas2}
	\xymatrix{ W[1,1,0] \ar_{\pi_{W[0,1,0]}^{W[1,1,0]}}[d] \ar^{\gamma \times Id_{S[1,1,0]}}[rr] && W'[1,1,0] \ar^{\pi_{W'[0,1,0]}^{W'[1,1,0]}}[d]  \\
	W[0,1,0] \ar_{\gamma}[rr] && W'[0,1,0]}.
\end{equation}

Indeed, by equation (\ref{eq:int-varphi}) we have 

$$ \gamma^{*}(\pi_{W'!}[\varphi]) = \gamma^* \left(\left(\pi_{W'!}^{W'[0,1,0]}\right)_{!} 
\left(\pi_{W'[0,1,0]}^{W'[1,1,0]}\right)_{!}\left[E(x)e(\xi)\otimes \delta'_{!}(p'^{*}[\varphi_0])\right]\right).$$

Then, it follows from lemma \ref{projlem} in the case of diagram (\ref{cas1}) that 

$$\gamma^{*}(\pi_{W'!}[\varphi]) =  \left(\pi_{W!}^{W[0,1,0]}\right)_{!} \left[ (\gamma \times Id_{S[0,1,0]})^*
\left(\left(\pi_{W'[0,1,0]}^{W'[1,1,0]}\right)_{!}\left[E(x)e(\xi)\otimes \delta'_{!}(p'^{*}[\varphi_0])\right]\right)
\right],
$$
then, it follows from lemma \ref{projlem} in the case of diagram (\ref{cas2}) that 
$$\gamma^{*}(\pi_{W'!}[\varphi]) =  \left(\pi_{W!}^{W[0,1,0]}\right)_{!} \left[ \left(\pi_{W[0,1,0]}^{W[1,1,0]}\right)_{!}
(\gamma \times Id_{S[1,1,0]})^* \left(\left[E(x)e(\xi)\otimes\delta'_{!}(p'^{*}[\varphi_0])\right]\right)\right],
$$
and finally, as $x$ and $\xi$ are coordinates in the bundle $W'[1,1,0]$ and are independent from $W'$ we have

$$\gamma^{*}(\pi_{W'!}[\varphi]) =  \left(\pi_{W!}^{W[0,1,0]}\right)_{!} \left(\pi_{W[0,1,0]}^{W[1,1,0]}\right)_{!}
\left[E(x)e(\xi) \otimes (\gamma \times Id_{S[1,1,0]})^* \left[\delta'_{!}(p'^{*}[\varphi_0])\right]
\right].
$$

Applying theorem \ref{mainthm} in the case without exponential (see remark \ref{remsansexp}) we obtain the equality
$$ (\gamma \times Id_{S[1,1,0]})^* \left(\delta'_{!}(p'^{*}[\varphi_0])\right) =  \delta_{!}\left((\gamma \times Id_{X[0,n,0]})^*(p'^{*}[\varphi_0])\right)$$
and by commutativity of the diagram we have 

$$ (\gamma \times Id_{S[1,1,0]})^* \left(\delta'_{!}(p'^{*}[\varphi_0])\right) =  \delta_{!}\left(p^*(\gamma \times Id_{X})^*[\varphi_0]\right)$$
and we can finally conclude

$$\gamma^{*}(\pi_{W'!}\varphi) =  \left(\pi_{W!}^{W[0,1,0]}\right)_{!}  \left(\pi_{W[0,1,0]}^{W[1,1,0]}\right)_{!} \left(
E(x)e(\xi) \otimes \delta_{!}\left[p^*(\gamma \times Id_{X})^*\varphi_0\right]
\right).
$$
namely by equation (\ref{eq:varphi} )
$$\gamma^{*}(\pi_{W'!}[\varphi]) = \pi_{W!}([(\gamma\times Id_X)^*\varphi]).$$

We conclude now by the proof of the projection lemma \ref{projlem} in the cases of diagrams (\ref{cas1}) and (\ref{cas2}). 
\begin{itemize}
	\item[$\bullet$] As in the paragraph \ref{case:S[0,n,0]} the case of diagram (\ref{cas1}) follows by the definition of the push-forward in the residue variables in \cite[\S 3.6]{CluLoe10a}.

	\item[$\bullet$] Taking pull-back of a cell decomposition as in the proof of the case without exponential (see section \ref{cas-sans-exp}), the case of diagram \ref{cas2} follows directly from the construction in \cite[\S 5.1, lemma 5.1.1]{CluLoe10a} (sketched in example \ref{example-integration-avec-exponentielle}) where the set of parameters is $W'[0,1,0]$ for the integration along $\pi_{W'[0,1,0]}^{W'[1,1,0]}$ and $W[0,1,0]$ for the integration along $\pi_{W[0,1,0]}^{W[1,1,0]}$. 
		Without giving the details, the constructions (sketched in example \ref{example-integration-avec-exponentielle}) 
		of $\pi_{W'[1,1,0]!}^{W'[0,1,0]}([\varphi])$ and 
		$\pi_{W[1,1,0]!}^{W[0,1,0]}((\gamma \times Id_{S[0,1,0]})^*[\varphi])$ for a $\pi_{W'[1,1,0]!}^{W'[0,1,0]}$-integrable-constructible function $\varphi$ in $\mathcal C(W'[1,1,0])^{\exp}$ are step by step compatible with the base change from $W'[0,1,0]$ to $W[0,1,0]$, and the projection formula 
		$$(\gamma \times Id_{W'[0,1,0]})^*\left(\pi_{W'[1,1,0]!}^{W'[0,1,0]}([\varphi])\right) =
		\pi_{W[1,1,0]!}^{W[0,1,0]}\left((\gamma \times Id_{S[0,1,0]})^*[\varphi]\right)$$
		follows from equation \ref{formule-integration-exp-valuee} using above ideas and the case without exponential. 
	\end{itemize}

\subsection*{Acknowledgments}
We are very grateful to Raf Cluckers, Immanuel Halupczok and Fran\c{c}ois Loeser for inspiring discussions. The first author was supported by the European Research Council under the European Community's Seventh Framework Programme (FP7/2007-2013) with ERC Grant Agreement no. 615722 MOTMELSUM, by the Labex CEMPI (ANR-11-LABX-0007-01). The first author would also like to thank VIASM in Hanoi for the hospitality and the great environment for work, during a three months visit in the summer 2018. The second author is partially supported by ANR-15-CE40-0008 (Defigeo) and by the ERC Advanced Grant NMNAG.

\bibliographystyle{plain}
\bibliography{biblio_complete}

\end{document}